\documentclass[11pt,reqno]{amsproc}

\usepackage{amsmath,amsfonts,amssymb,amsthm}
\usepackage[abbrev,lite,nobysame]{amsrefs}
\usepackage{mathrsfs}
\usepackage{graphics,graphicx}
\usepackage[usenames,dvipsnames]{color}
\usepackage{times}
\usepackage{bbm}
\usepackage[margin=1in]{geometry}

\usepackage[colorlinks=true, pdfstartview=FitV, linkcolor=BrickRed,citecolor=black, urlcolor=black]{hyperref}

\usepackage{mathtools}
\mathtoolsset{showonlyrefs}

\newtheorem{theorem}{Theorem}
\newtheorem{proposition}{Proposition}[section]
\newtheorem{lemma}[proposition]{Lemma}

\theoremstyle{definition}

\newtheorem{remark}[proposition]{Remark}

\numberwithin{equation}{section}

\newcommand\eps{\varepsilon}
\newcommand\e{{\rm e}}
\newcommand\dd{{\rm d}}
\newcommand\ddt{{\frac{\dd}{\dd t}}}

\def\l {\langle}
\def\r {\rangle}
\newcommand\de{{\partial}}

\newcommand\cA{{\mathcal A}}

\newcommand\cL{{\mathcal L}}
\newcommand\cM{{\mathcal M}}
\newcommand\cN{{\mathcal N}}
\newcommand\cP{{\mathcal P}}

\newcommand\cT{{\mathcal T}}

\newcommand{\norm}[1]{\left\lVert #1 \right\rVert}

\newcommand{\abs}[1]{\left\lvert #1 \right\rvert}

\makeatletter

\renewcommand\subsubsection{\@startsection{subsubsection}{3}%
\normalparindent{.5\linespacing\@plus.7\linespacing}{-.5em}
{\normalfont\bfseries}}

\def\@tocline#1#2#3#4#5#6#7{\relax
  \ifnum #1>\c@tocdepth 
  \else
    \par \addpenalty\@secpenalty\addvspace{#2}%
    \begingroup \hyphenpenalty\@M
    \@ifempty{#4}{%
      \@tempdima\csname r@tocindent\number#1\endcsname\relax
    }{%
      \@tempdima#4\relax
    }%
    \parindent\z@ \leftskip#3\relax \advance\leftskip\@tempdima\relax
    \rightskip\@pnumwidth plus4em \parfillskip-\@pnumwidth
    #5\leavevmode\hskip-\@tempdima
      \ifcase #1
       \or\or \hskip 1em \or \hskip 2em \else \hskip 3em \fi%
      #6\nobreak\relax
    \dotfill\hbox to\@pnumwidth{\@tocpagenum{#7}}\par
    \nobreak
    \endgroup
  \fi}
\makeatother

\newcommand{\NN}{\mathbb{N}}
\newcommand{\ZZ}{\mathbb{Z}}
\newcommand\TT {{\mathbb T}}
\newcommand\RR {{\mathbb R}}

\newcommand\m{\mathfrak{m}}

\newcommand{\PP}{\mathbb{P}}

\newcommand\tQ{{\widetilde Q}}
\newcommand\tW{{\widetilde W}}
\newcommand\cO{{\mathcal O}}
\newcommand\fm{{\mathfrak{m}}}

\begin{document}

\title[On the stability of viscous three-dimensional rotating Couette flow]{\vspace*{-2.5cm}
On the stability of viscous three-dimensional rotating Couette flow}

\author[M. Coti Zelati]{Michele Coti Zelati}
\address{Department of Mathematics, Imperial College London}
\email{m.coti-zelati@imperial.ac.uk}

\author[A. Del Zotto]{Augusto Del Zotto}
\address{Institute of Mathematics, University of Zurich}
\email{augusto.delzotto@math.uzh.ch}

\author[K. Widmayer]{Klaus Widmayer}
\address{Faculty of Mathematics, University of Vienna \& Institute of Mathematics, University of Zurich}
\email{klaus.widmayer@math.uzh.ch}

\subjclass[2020]{35Q35, 76D05, 76E07, 76U60}

\begin{abstract}

   We study the stability of Couette flow in the $3d$ Navier-Stokes equations with rotation, as given by the Coriolis force. Hereby, the nature of linearized dynamics near Couette flow depends crucially on the force balance between background shearing and rotation, and includes lift-up or exponential instabilities, as well as a stable regime. In the latter, shearing resp.\ rotational inertial waves give rise to mixing and dispersive effects, which are relevant for distinct dynamical realms. Our main result quantifies these effects through enhanced dissipation and dispersive amplitude decay in both linear and nonlinear settings: in particular, we establish a nonlinear transition threshold which quantitatively improves over the setting without rotation (and increases further with rotation speed), showcasing its stabilizing effect.

\end{abstract}

\maketitle

\setcounter{tocdepth}{2}
\tableofcontents

\section{Introduction}

The subject of this work are dynamics in the $3d$ Navier-Stokes equations as witnessed in a rotating frame of reference, i.e.\ solutions of the Navier-Stokes-Coriolis system
\begin{equation}\label{eq:NSC}
    \begin{cases}
        \de_tV+V\cdot\nabla V + \beta \vec\e_3\times V-\nu\Delta V +\nabla P=0,\\
        \nabla\cdot V=0,
    \end{cases}
\end{equation}
whereby $V:[0,\infty)\times\TT\times\RR\times\TT\to\RR^3$ is the incompressible ($\nabla\cdot V=0$) fluid velocity, $P:[0,\infty)\times\TT\times\RR\times\TT\to\RR$ the scalar pressure and we have chosen coordinates such that the axis of rotation is $\vec{e}_3$. The two parameters of the model are the speed of rotation $\beta$, reflected in the Coriolis force $\beta\vec{e}_3\times V$, and the viscosity $\nu>0$.
Our main result investigates the nonlinear stability of Couette flow\footnote{Rescaling arguments allow to treat the case of Couette flow with any (constant) slope, hence we are restricting here to the unit slope case -- see point \eqref{it:rescalings} of Remark \ref{rem:linear}.}
\begin{equation}
    V^s=(y,0,0), \qquad \de_yP^s=-\beta y,
\end{equation}
which is a stationary solution to \eqref{eq:NSC}. To understand the dynamics of perturbations, we write $V=V^s+u$ and obtain from \eqref{eq:NSC} that $u:[0,\infty)\times\TT\times\RR\times\TT\to\RR^3$ must satisfy
\begin{equation}\label{eq:perturb_system}
    \begin{cases}
        \de_t u+y\de_x u+ u^2\vec\e_1 + \beta \vec\e_3\times u  -\nu\Delta u +\nabla p^L= -(u\cdot\nabla) u -\nabla p^{NL},\\
        \nabla\cdot u=0,
    \end{cases}
\end{equation}
where the linear and nonlinear pressure terms are
\begin{align}
    \Delta p^L &=-2\de_xu^2 +\beta(\de_xu^2-\de_yu^1),\qquad \Delta p^{NL}= -\nabla\cdot(u\cdot\nabla u). 
\end{align}
As is clear from the form of the equations \eqref{eq:perturb_system}, the general dynamics of solutions is different from those that are independent of the periodic $x$- resp.\ $z$-variables. To take this into account, for a function $F:\TT\times\RR\times\TT\to\mathcal{V}$, $\mathcal{V}\in\{\RR,\RR^3\}$, we introduce the notations
\begin{equation}\label{eq:modes_notation}
\begin{aligned}
  \overline{F}_0(y)&=\frac{1}{4\pi^2}\int_{\TT^2}F(x,y,z)\dd x\,\dd z,\qquad \widetilde{F}_0(y,z)=\frac{1}{2\pi}\int_{\TT}F(x,y,z)\dd x -\overline{F}_0(y),\\
  F_{\neq}(x,y,z)&=F(x,y,z)-\frac{1}{2\pi}\int_{\TT}F(x,y,z)\dd x,
\end{aligned}  
\end{equation}
and refer to $\overline{F}_0$ as ``double zero'', $\widetilde{F}_0$ as ``simple zero'' and $F_{\neq}$ as ``non-zero'' modes of $F$. 

As a first step, we discuss the behavior of solutions to the linearization of \eqref{eq:perturb_system}, i.e.\ of
\begin{equation}\label{eq:linearized}
        \de_t u+y\de_x u+ u^2\vec\e_1 + \beta \vec\e_3\times u  -\nu\Delta u +\nabla p^L=0,\qquad \nabla\cdot u=0.
\end{equation} 
The key quantity hereby is the Bradshaw-Richardson\footnote{sometimes also called Pedley number \cites{P69,LC97}} number (see e.g.\ \cites{B69,H18})
\begin{equation}\label{eq:BR-number}
    B_\beta:=\beta(\beta-1)\in [-1/4,\infty).
\end{equation}
It captures and quantifies the stable resp.\ unstable nature of dynamics in \eqref{eq:linearized}:
\begin{theorem}[Linear dynamics]\label{thm:linear}
Let $u\in C_t([0,\infty);L^2(\TT\times\RR\times\TT))$ solve \eqref{eq:linearized}. Then 
\begin{equation}\label{eq:double_zero_lin}
   \overline{u}^2_0(t)=0,\qquad \overline{u}^j_0(t)=\e^{\nu t\de_{yy}}\overline{u}^{j}_0(0),\quad j\in\{1,3\}, 
\end{equation}
and we have the following behavior depending on $B_\beta$:
    \begin{enumerate}
        \item\label{it:linear_liftup} (Lift-up instabilities for simple zero modes when $B_\beta=0$) When $\beta=0$, the well-known lift-up instability occurs for $\widetilde{u}_0$ (see e.g.\ \cite{CWZ20}*{Section 2.1}), while for $\beta=1$ the simple zero modes witness a ``rotated'' lift-up instability:
        \begin{equation}\label{eq:liftup_lin}
          \widetilde{u}^1_0(t)=\e^{\nu t\Delta} \widetilde{u}^{1}_0(0),\quad \widetilde{u}^2_0(t)=\e^{\nu t\Delta}  \widetilde{u}^{2}_0(0)-t\e^{\nu t\Delta}\de_z^2\Delta^{-1}\widetilde{u}^{1}_0(0),\quad \widetilde{u}^3_0(t)=-\de_z^{-1}\de_y\widetilde{u}^2_0(t).
        \end{equation}
        \item\label{it:linear_unstable} (Exponential instabilities for simple zero modes when $B_\beta<0$, i.e.\ $0<\beta<1$) If $\sqrt{-B_\beta}>\nu$, a part of the simple zero modes undergoes an exponential norm inflation. More precisely, there exists an $L^2$-bounded orthogonal projection $\PP\neq 0$ such that
        \begin{equation}
            \norm{\PP\,\widetilde{u}^j_0(t)}_{L^2}\gtrsim \e^{\frac12(\sqrt{-B_\beta}-\nu) t}\norm{\PP\,\widetilde{u}_0(0)}_{L^2},\qquad j\in\{1,2,3\}.
        \end{equation}
        
        \item\label{it:linear_stable} (Stable dynamics when $B_\beta>0$, i.e.\ $\beta<0$ or $\beta>1$) We have enhanced dissipation and inviscid damping of the non-zero modes
        \begin{equation}\label{eq:main_lin_enhdiss}
            \norm{u^1_{\neq}(t)}_{L^2}+\l t\r\norm{u^2_{\neq}(t)}_{L^2} +\norm{u^3_{\neq}(t)}_{L^2}\lesssim c_\beta \e^{-\frac{1}{24}\nu t^3}\norm{u_{\neq}(0)}_{H^{2}},
        \end{equation}
        and dispersion of the simple zero modes
        \begin{equation}\label{eq:main_lin_disp}
            \norm{\widetilde{u}^j_0(t)}_{W^{2,\infty}}\lesssim c_\beta (\alpha t)^{-\frac13}\e^{-\nu t}\norm{\widetilde{u}_0(0)}_{W^{6,1}}, \qquad j\in\{1,2,3\},
        \end{equation}
        where $\alpha:=\sqrt{B_\beta}>0$ and $c_\beta:=\sqrt{\frac{(\beta-1)^2+\beta^2}{B_\beta}}>0$.
        
    \end{enumerate}
\end{theorem}

A more detailed description of these dynamics (including in particular the regimes with instabilities) can be found in Section \ref{sec:linear}, where the proof of Theorem \ref{thm:linear} is given via direct computations.
\begin{remark}\label{rem:linear}
We comment briefly on some key points:
\begin{enumerate}
    \item One sees directly that the double zero modes in \eqref{eq:linearized} follow a purely heat equation dynamic $\de_t \overline{u}^j_0-\nu\de_{yy}\overline{u}^j_0=0$, $j=1,3$, while $\overline{u}_0^2=0$ if $u$ is divergence-free and integrable, leading to \eqref{eq:double_zero_lin}.
    \item The simple zero modes witness a coupling of the equations due to the background shearing and rotation, but no transport effects. In particular, they satisfy a constant coefficient system for two (suitable) scalar unknowns, the eigenvalues of which lead to the dynamics described above. In case $\beta=1$, \eqref{eq:liftup_lin} follows directly. Otherwise, after a Fourier transform one can diagonalize this system to find the eigenvalues for a frequency $(0,\eta,l)\in\ZZ\times\RR\times\ZZ$ given by 
    \begin{equation}\label{eq:disp_relation}
    \begin{cases}
        -\nu\abs{\eta,l}^2\pm\sqrt{\beta(1-\beta)}\abs{l}\abs{\eta,l}^{-1}, &0<\beta<1,\\
        -\nu\abs{\eta,l}^2\pm i\sqrt{\beta(\beta-1)}\abs{l}\abs{\eta,l}^{-1}, &\beta<0\textnormal{ or } \beta>1.
    \end{cases}    
    \end{equation}
    This shows the potential exponential instability for $0<\beta<1$\footnote{The projection $\PP$ above simply corresponds to restricting to an appropriate set of frequencies.} and the oscillatory nature of the coupling between shearing and rotation in the stable regime $\beta<0$ or $\beta>1$. (We note that both effects affect all three components of $\widetilde{u}_0$.)
    \item\label{it:lin_mixing} The mixing effect of the transport term $\de_t+y\de_x$ is only reflected in the non-zero modes, which exhibit enhanced dissipation and inviscid damping. We state this quantitatively \eqref{eq:main_lin_enhdiss} only in the stable regime $B_\beta>0$, for which we also pursue a nonlinear result. (The analogous claim remains true in case $B_\beta=0$, see \cites{WZ21,HSX24dec}, and one may also hope for a qualitatively similar effect at least for the stable part of the dynamics in case $B_\beta< 0$.) On a technical level, the delicate interplay of these two effects is captured via a suitable multiplier $\m$, introduced in \cite{BGM17} and discussed later in Section \ref{ssec:nonzero_linear}. As compared to the case without rotation ($\beta=0$), the inviscid damping rate of $u^2_{\neq}$ in \eqref{eq:main_lin_enhdiss} is one order slower. This can be traced back to the coupling of different components of $u$ due to the Coriolis term $\beta\vec{e}_3\times u$ -- see the proof of \eqref{eq:main_lin_enhdiss} in Section \ref{ssec:nonzero_linear} below for more details.
    \item\label{it:Bbeta_singlimit} We highlight that the limit $B_\beta\searrow 0$ is singular for our estimates in that the bounds \eqref{eq:main_lin_enhdiss} and \eqref{eq:main_lin_disp} degenerate, since $c_\beta\to\infty$ as $B_\beta\searrow 0$. In contrast, $c_\beta$ is uniformly bounded when $B_\beta\to\infty$. Moreover, as can be seen from \eqref{eq:main_lin_disp}, the effective strength of rotation is given by $\alpha$. In the regime of fast rotation $\beta\gg 1$ one has $\alpha\sim\abs{\beta}$ and the bound \eqref{eq:main_lin_disp} improves with increasing $\alpha$: This quantifies the stabilizing effect of rotation, whereby a faster speed of rotation $\beta$ leads to more rapid decay of amplitudes of simple zero modes.
    \item\label{it:rescalings} By rescaling one sees that the scenario of dynamics near a more general Couette flow with slope $\sigma\neq 0$, i.e.\ $V^s_\sigma=(\sigma y,0,0)$, can be reduced to the present one: if $u_\sigma$ follows the linearized dynamics near $V^s_\sigma$, then $u(t,x,y,z):=\sigma^{-1} u_\sigma(\sigma^{-1} t,x,y,z)$ satisfies \eqref{eq:linearized} with $\beta\mapsto\sigma^{-1}\beta$, $\nu\mapsto\sigma^{-1}\nu$. In particular, this shows that increasing $\sigma$ (while keeping $\beta$ fixed) decreases the Bradshaw-Richardson number. When $\beta$ and $\sigma$ have the same sign, this can destabilize the flow since $B_{\sigma^{-1}\beta}<0$ for $\abs{\sigma}>\abs{\beta}$, while even when stability is maintained, amplitude decay is slowed down (whereas enhanced dissipation rates are increased according to the time-scale $\sigma t$). One checks directly that these arguments also apply for the nonlinear problem.\footnote{Moreover, at the cost of rescaling the spatial variables one can also keep $\nu$ fixed (e.g.\ considering $\sigma^{-\frac12}u_\sigma(\sigma^{-1}t,\sigma^{-\frac12}x,\sigma^{-\frac12}y,\sigma^{-\frac12}z)$), which gives another equivalence at least at fixed regularity.}

\end{enumerate}
\end{remark}

Our main result is a new nonlinear transition threshold in the (linearly) stable regime ($B_\beta>0$, see Theorem \ref{thm:linear}), below which essential features of the linear dynamics persist:
\begin{theorem}[Nonlinear stability]\label{thm:transition threshold}
    Let $N>3$ and $0<\nu<1$. There exist constants $c_1,c_2>0$ such that the following holds true:
    Let $u(0)$ be an initial datum for \eqref{eq:perturb_system} with
    \begin{equation}\label{eq:initial_data}
        \overline{u}_0(0)=\int_{\TT^2} u(0)= 0,\qquad \norm{u(0)}_{H^{N+2}\cap W^{6,1}}=:\eps>0.
    \end{equation}
    If 
    \begin{equation}\label{eq:transition_thm}
        B_\beta>0 \text{ and }\eps \leq c_1\nu^\frac89, \quad \text{or} \quad B_\beta>c_2\nu^{-1}\text{ and }\eps\leq c_1\nu^\frac56,
    \end{equation}
    there exists a unique global in time solution $u\in C_t\left([0,\infty);H^N(\TT\times\RR\times\TT)\right)$ of \eqref{eq:perturb_system}, which moreover satisfies the following:
    \begin{enumerate}
        \item Energy inequality:
        \begin{equation}\label{eq:energyineq_thm}
            \norm{u}_{L^\infty_t H^N}+\nu^{\frac16}\norm{u_{\neq}}_{ L^2_tH^N}+\nu^{\frac12}\norm{\nabla u}_{ L^2_tH^N}\lesssim C_\beta \eps.
        \end{equation}
        \item Inviscid damping and enhanced dissipation:
        \begin{equation}\label{eq:inviscid+enhanced_thm}
             \norm{u^1_{\neq}(t)}_{L^2}+\l t \r\norm{u^2_{\neq}(t)}_{L^2} +\norm{u^3_{\neq}(t)}_{L^2}\lesssim C_\beta \e^{-\frac{1}{24} \nu^{\frac13}t}\eps. 
        \end{equation}
        \item Dispersion: letting $\alpha:=\sqrt{B_\beta}$,
        \begin{equation}\label{eq:disp_thm}
            \|\widetilde{u}_{0}^j(t)\|_{W^{2,\infty}}\lesssim C_\beta ((\alpha t)^{-\frac13}\e^{-\nu t}+\alpha^{-\frac13}\nu^{-\frac23}\eps)\eps, \qquad j\in\{1,2,3\}. 
        \end{equation}
    \end{enumerate}
    Here, $C_\beta>0$ is a universal constant that is uniformly bounded for $B_\beta \geq 0.1$, but that degenerates $C_\beta\to\infty$ as $B_\beta\searrow 0$.
\end{theorem}
We highlight that compared to the case without rotation ($\beta=0$) where the best known threshold is of order $\cO(\nu)$ \cite{WZ21}, here we obtain a higher threshold of order $\nu^{\frac89}\gg\nu$ which increases further with increasing speed of rotation \eqref{eq:transition_thm}, a clear and quantitative indication of the stabilizing effect of rotation. Without using this stabilizing effect, prior work \cite{HSX24sept} had established an $\cO(\nu)$ threshold also for \eqref{eq:perturb_system}.

Theorem \ref{thm:transition threshold} follows from a bootstrap argument that we discuss in Section \ref{sec:bootstrap} along with an overview of its proof, after having formulated the equations via suitable ``good unknonws'' in Section \ref{sec:proofs_setup}. We briefly comment on some further points of immediate relevance in Theorem \ref{thm:transition threshold}.

\begin{remark}\label{rem:nonlinear}
Several remarks are in order:
\begin{enumerate}
    \item\label{it:00-mean-vanish} The vanishing of the $x$-$z$-mean $\overline{u}_0(0)=0$ of $u(0)$ is a natural assumption: by \eqref{eq:double_zero_lin} the double zero modes $\overline{u}_0(t)$ obey a nonlinear heat equation in one variable, which in general precludes a nonlinear stability threshold larger than $\cO(\nu)$. However, exact vanishing is not necessary and this condition could be relaxed to a suitably quantified smallness consistent with the bounds in Proposition \ref{prop:doublezero_estimates}.

    \item The choice of a doubly periodic channel domain $\TT\times\RR\times\TT$ provides a simple setting to study shearing and rotational effects in the absence of boundaries  -- see also \cite{TTA10} and references therein for experimental results.

    \item In general it is expected that the transition threshold is quantified in terms of relevant parameters, which in our setting are the inverse Reynolds number $\nu$ and the Bradshaw-Richardson number $B_\beta$ -- see \eqref{eq:transition_thm}. While fast rotation thus improves the stability of the system, we note that its effect on the non-zero modes remains to be understood.
    
    \item In the linearly unstable regime $B_\beta\leq 0$, nonlinear transition thresholds are only known in the case of lift-up instabilities (i.e.\ $B_\beta=0$): For the non-rotating case $\beta=0$, this is given in \cite{WZ21} by $\cO(\nu)$ (which is conjectured to be sharp in general), whereas the setting of $\beta=1$ with the ``rotated'' lift-up instability \eqref{eq:liftup_lin} was recently shown to be stable at least to perturbations of order up to $\cO(\nu^2)$. In case of exponential instabilities $B_\beta<0$, numerical simulations \cites{YFMR92one,MFYRL95two} indicate the persistence of instabilities at the nonlinear level, but this is not rigorously understood.
\end{enumerate}
\end{remark}

\subsubsection*{Context}
A quantitative understanding of the stability of laminar flows is a central question in hydrodynamic stability theory, originating in the works of Rayleigh, Reynolds, and Kelvin \cites{rayleigh1879stability,reynolds1883xxix,kelvin1887stability}. While turbulence and its onset remain largely ill-understood, recent years have seen a lot of progress in the study of stable dynamics near shear flows in the Navier-Stokes  or Euler equations \cites{BM15, IJnon20, MZ20,CDLZ23,BGM17,BGM20,BGM22,CWZ20,WZ21,LWZ20Kolmo,BMV16,BVW18,DL23,WZZKolmo20,WZ23,DelZotto23,DL22,CZEW20,CLWZ20,MZ22}. In particular, linear mixing mechanisms near monotone shears are relatively well-understood and quantified. They have also been exploited to construct nonlinearly stable dynamics near stationary shears. The size of their basin of attraction (also known as the transition threshold) is hereby naturally quantified in terms of physically relevant parameters (in particular the Reynolds number), as already observed by Reynolds. 

The related literature is too vast to be surveyed here (see e.g.\ \cite{BGM19} and references therein for some overview), so we shall focus first on the setting of the $3d$ Navier-Stokes equations with high Reynolds number, proportional to $\nu^{-1}\gg 1$, i.e.\ \eqref{eq:NSC} with $\beta=0$. 
A classical and comparatively simple example of a stationary shear solution for these equations is Couette flow, near which the linearized dynamics can be computed explicitly.
On various partially or fully periodic domains and in Sobolev or Gevrey regularity, its nonlinear stability has been shown for perturbations of order $\cO(\nu)$ -- see \cites{BGM19,BGM20,BGM22,WZ21,CWZ20}.
The key mixing effect hereby is enhanced dissipation: this provides stability through decay of $L^2$-based norms of the non-zero modes $V_{\neq}$ on a timescale of order $\nu^{-\frac13}$, which is much shorter than the purely dissipative time scale $\nu^{-1}$. This is accompanied by inviscid damping, leading to a quadratic in time decay of $V^2_{\neq}$ in $L^2$-based spaces. Altogether this shows a meta-stable dynamic towards simple or double zero modes.
Due to an intrinsically $3d$, linear instability known as the lift-up effect, a transient growth of order $\nu^{-1}$ in the simple zero modes is indeed expected, so that the above stability results are conjectured to be sharp.

\subsubsection*{The effect of rotation.} In a rotating frame of reference, fluid motion is additionally subject to the Coriolis force, which is generally regarded as a stabilizing effect (see e.g.\ \cite{GS2007} for a review of rotational effects in the context of geophysical flows), and the Navier-Stokes-Coriolis system \eqref{eq:NSC} is thus a natural yet basic model. An effect of rotation in \eqref{eq:NSC} can then be seen in the linearized dynamics near the rest state $V=0$, which is given by dispersive inertial waves with dispersion relation $\Lambda$ satisfying $\Lambda^2(k,\eta,l)=\abs{l}^2\abs{k,\eta,l}^{-2}$, with $(k,\eta,l)\in\ZZ\times\RR\times\ZZ$. Depending on the geometry of the domain, these lead to amplitude decay or spatial averaging, thereby stabilizing the fluid motion. In particular, this can be used to construct (nonlinear) global solutions (in the purely periodic or Euclidean setting, see e.g.\ \cites{BMN00,CDGG06,GRM09,IT2013,ET2023}), provided the speed of rotation is sufficiently fast compared to the size of the initial data.

Thanks to counter-balancing pressure forces, a suitably oriented Couette flow is a stationary solution also in the presence of rotation \eqref{eq:NSC}. However, the resulting interplay of shearing and rotational forces on perturbations is surprisingly rich already at the linearized level, as we discuss in Theorem \ref{thm:linear}: depending on the relative strength of these effects as quantified by the Bradshaw-Richardson number $B_\beta$ in \eqref{eq:BR-number}, the behavior of simple zero modes\footnote{These do not witness the transport effect of the shearing background and satisfy a constant-coefficient system, see \eqref{eq:linearsystem-QK-simplezero}, which makes them natural analogues for the behavior of perturbations of the rest state.} includes exponential instabilities, a ``rotated'' lift-up instability as well as the aforementioned dispersive inertial waves. As commented on above, our main result exploits the latter effect via the amplitude decay it engenders for our domain, quantified via \eqref{eq:main_lin_disp} resp.\ \eqref{eq:disp_thm}. (We remark, however, that a full quantification of the combined effect of rotation and shearing also on the non-zero modes is still outstanding.) The diversity of these dynamics has also been explained with heuristics from physics in the so-called \emph{displaced particle approach} introduced by Tritton and Davies (see \cites{Tritton1985,Tritton92} for a more in-depth description) and refined later by Leblanc and Cambon \cite{LC97} to highlight the quintessentially three-dimensional nature of this problem.

\subsubsection*{A parallel with inhomogeneous fluids.}
In the geophysical fluids literature, parallels are often drawn between the effects of rotation and stratification (see e.g.\ \cites{V1970,F1980-10,B2014}). Roughly speaking, rotational Coriolis resp.\ buoyant forces tend to organize and stabilize fluid flows perpendicular to the direction of rotation resp.\ gravity by inhibiting motion in these directions. The underlying mechanisms -- at least in the linearized theory -- hereby are dispersive inertial resp.\ gravity waves with closely related, zero-homogeneous dispersion relations such as $\Lambda$ (see also our discussion in \cite{CZDZW24} and related works \cites{TS17,KLT2014,W19,GPW23}). As already commented on by Bradshaw in \cite{B69}, these parallels endure in the setting of linearized dynamics near (Couette) shear flow in the Navier-Stokes-Coriolis resp.\ $3d$ Boussinesq equations: in both settings linearly stable dynamics can be identified via physically relevant parameters, here the Bradshaw-Richardson number $B_\beta>0$ (and analogously the Richardson number for stably stratified flow).\footnote{In particular, we note that lift-up instabilities are suppressed in the stable regimes, see also \cite{CZDZ23}. This is witnessed also in \cite{L18} for the MHD setting, when the $3d$ Couette flow is coupled with a constant background magnetic field. However, the analogous oscillatory dynamics in this case are not dispersive.} However, due to the background shear flow, the overall dynamics are much richer -- see e.g.\ Theorem \ref{thm:linear} and the discussion thereafter. Together with our previous work \cite{CZDZW24}, the present article gives a quantitative validation to these heuristics also for the nonlinear dynamics: in both cases, we establish an improvement of the stability theory over the case without rotation resp.\ stratification, obtaining a threshold of at least $\cO(\nu^\gamma)$ for some $\gamma<1$. The key to this is the exploitation of the aforementioned dispersive nature of certain oscillations in the simple zero modes.

\subsection{Structure of the equations and choice of unknowns}\label{sec:proofs_setup}
As is clear from \eqref{eq:perturb_system}, the double zero modes evolve according to a particular simple dynamic: while $\overline{u}^2_0=0$ since $u$ is divergence-free and square-integrable, $\overline{u}^j_0$, $j=1,3$, obey the following nonlinear $1d$ heat equations on $\RR$:
\begin{equation}\label{eq:doublezero}
    \partial_t \overline{u}^j_0-\nu\de_y^2 \overline{u}^j_0=-\overline{\de_y(u^2u^j)}_0, \quad j=1,3.
\end{equation}
Rotation and shearing are thus only witnessed in the simple zero and non-zero modes. To understand their evolution, it is convenient to consider the unknowns
\begin{equation}\label{eq:qomega}
    q:=\Delta u^2,\qquad \omega^2:=\de_zu^1-\de_xu^3,
\end{equation}
from which, thanks to the incompressibility condition, all but the double zero modes of $u$ can be recovered as
\begin{equation}\label{eq:u_recover}
    \begin{aligned}
    u^1-\overline{u}^1_0 &=(\de_x^2+\de_z^2)^{-1}(\de_z\omega^2-\de_{xy}\Delta^{-1}q),\\
    u^2&=\Delta^{-1}q,\\
    u^3-\overline{u}^3_0&=-(\de_x^2+\de_z^2)^{-1}(\de_x\omega^2+\de_{yz}\Delta^{-1}q).
    \end{aligned}
\end{equation}
To account for the shearing effect of the Couette flow, it is natural to consider moving frame coordinates 
\begin{equation}\label{eq:moving_frame}
    (x,y,z)\mapsto (x-yt,y,z),
\end{equation}
along which differential operators transform as
\begin{equation}
    \nabla \mapsto \nabla_L=(\de_1,\de_2^L,\de_3)=(\de_x,\de_y^L,\de_z)=(\de_x,\de_y-t\de_x,\de_z), \quad \Delta\mapsto\Delta_L=\de_x^2+(\de_y^L)^2+\de_z^2.
\end{equation}
To highlight this perspective in the notation, we will use the convention of denoting functions in the moving frame with capitalized letters, i.e.\ $f\mapsto F$.

The linearized system \eqref{eq:linearized} in the (moving frame) for the variables $Q,\Omega^2$ then reads
\begin{equation}\label{eq:linearsystem-QOmega_0}
    \begin{cases}
        \de_t Q+\beta\de_z \Omega^2=\nu\Delta_LQ,\\
        \de_t\Omega^2-(\beta-1)\de_z\Delta_L^{-1}Q=\nu\Delta_L\Omega^2,
    \end{cases}
\end{equation}
which together with the linear heat equation $\partial_t \overline{u}^j_0=\nu\de_y^2 \overline{u}^j_0$ for $\overline{u}^j_0$, $j=1,3$ (compare \eqref{eq:doublezero}) and \eqref{eq:u_recover} is equivalent to \eqref{eq:linearized}. We will use this formulation in the regime $B_\beta\leq 0$, and parts \eqref{it:linear_liftup} and \eqref{it:linear_unstable} of Theorem \ref{thm:linear} follow from direct computations (see Section \ref{sec:instable regime}).

In the stable regime $B_\beta>0$, we will work with the pair of ``good unknowns'' $(Q,W)$, where $W$ is a rescaled version of $\Omega^2$, that matches the regularity of $Q$ and balances the parameters in the equation, so as to highlight the oscillatory nature of the coupling between $Q$ and $W$:\footnote{One sees in particular that $\widetilde{Q}_0,\widetilde{W}_0$ are the natural unknowns for the constant coefficient system of the simple zero modes -- see also \eqref{eq:linearsystem-QK-simplezero}.}
\begin{equation}
    Q=\Delta_LU^2,\qquad W:=\sqrt{\frac{\beta}{\beta-1}}|\nabla_L|\Omega^2.
\end{equation}
Altogether, in this formulation and for $B_\beta>0$, the system \eqref{eq:perturb_system} is equivalent to 
\begin{equation}\label{eq:perturb_sys_nl}
    \begin{cases}
        \de_t Q +\alpha \de_z|\nabla_L|^{-1}W -\nu \Delta_L Q = -\de_y^L\Delta_L\mathcal{P}(U,U)-\Delta_L\mathcal{T}(U,U^2),\\
        \de_t W-\de_{x}\de_y^L\abs{\nabla_L}^{-2}W +\alpha \de_z|\nabla_L|^{-1}Q-\nu\Delta_L W=\sqrt{\frac{\beta}{\beta-1}}\abs{\nabla_L}\left(\de_z\mathcal{T}(U,U^1)-\de_x\mathcal{T}(U,U^3)\right),\\
        \de_t\overline{F}_0 -\nu\de_y^2\overline{F}_0=-\de_y\overline{(\widetilde{U}^2_0\widetilde{F}_0)}_0-\de_y\overline{(U^2_{\neq} F_{\neq})}_0,\qquad F\in\{U^1,U^3\},
    \end{cases}
\end{equation}
where
\begin{equation}
 \mathcal{T}(U,U^j)=U\cdot\nabla_LU^j,\quad 1\leq j\leq 3,\qquad \mathcal{P}(U,U)=-\Delta_L^{-1}(\nabla_L\otimes\nabla_L)(U\otimes U).   
\end{equation}

\subsection{Bootstrap for Theorem \ref{thm:transition threshold}}\label{sec:bootstrap}
The proof of Theorem \ref{thm:transition threshold} builds on the linear theory established in Theorem \ref{thm:linear}\eqref{it:linear_stable} and proceeds by establishing a bootstrap argument for energy type estimates in the formulation \eqref{eq:perturb_sys_nl} of the equations. These energies involve time and frequency weights combined in the main Fourier multiplier $\cA$, which is defined as
\begin{equation}\label{eq:A-prelim}
    \cA=\m\cM \e^{\delta\nu^{\frac13}t},
\end{equation}
see Section \ref{sec:prelim} for the details. Our main argument is the following:

\begin{theorem}[Bootstrap]\label{thm:bootstrap}
Under the hypotheses of Theorem \ref{thm:transition threshold}, assume that for some $T>0$ the following bounds hold for $t\in[0,T]$.
\begin{enumerate}
\item\label{it:00-btstrap} The double zero modes satisfy
\begin{equation}\label{eq:00-btstrap}
\begin{aligned}
    \norm{\overline{U}^1_0}^2_{L^\infty_tH^N}+\nu\norm{\de_y\overline{U}^1_{0}}^2_{L^2_tH^N}\leq 100\eps^2,\\
     \norm{\overline{U}^3_0}^2_{L^\infty_tH^{N}}+\nu\norm{\de_y \overline{U}^3_0}^2_{L^2_tH^{N}}\leq 100\eps^2. 
\end{aligned}
\end{equation}

\item\label{it:s0-btstrap} The simple zero modes satisfy
\begin{equation}\label{eq:s0-btstrap}
\begin{aligned}
   \norm{\tQ_{0}}^2_{L^\infty_tH^N}+\nu\norm{\nabla \tQ_{0}}^2_{L^2_tH^N}\leq 100\eps^2,\\
     \norm{\tW_{0}}^2_{L^\infty_tH^N}+\nu\norm{\nabla \tW_{0}}^2_{L^2_tH^N}\leq 100\eps^2. 
\end{aligned}
\end{equation}

\item\label{it:non0-btstrap} The nonzero modes satisfy
\begin{align}
     \norm{\cA Q_{\neq}}^2_{L^\infty_tH^N}+\nu\norm{\nabla_L\cA Q_{\neq}}^2_{L^2_tH^N}+\norm{\sqrt{-\frac{\dot \cM}{\cM}}\cA Q_{\neq}}^2_{L^2_tH^N}\leq 100\eps^2,\label{eq:btstrap_Q_neq}\\
     \norm{\cA  W_{\neq}}^2_{L^\infty_tH^N}+\nu\norm{\nabla_L\cA W_{\neq}}^2_{L^2_tH^N}+\norm{\sqrt{-\frac{\dot \cM}{\cM}}\cA W_{\neq}}^2_{L^2_tH^N}\leq 100\eps^2.\label{eq:btstrap_W_neq}
\end{align}
\end{enumerate}
Then the same bounds hold with $100$ replaced by $50$.
\end{theorem}
Hereby, as stated in Theorem \ref{thm:transition threshold} implicit constants may degenerate as $B_\beta\searrow 0$ (this is apparent already from the choice of unknowns, see also Remark \ref{rem:linear}\eqref{it:Bbeta_singlimit}), but since our focus is not on this limit but rather a regime including large $\abs{\beta}$ we have chosen to suppress this dependence from the notation and will only track constants in so far as they are not uniformly bounded for (say) $B_\beta\geq 0.1$.

Theorem \ref{thm:transition threshold} follows directly from this bootstrap: From standard well-posedness theory we obtain a local in time solution to \eqref{eq:perturb_sys_nl} that satisfies the assumptions of Theorem \ref{thm:bootstrap} on some time interval $[0,T]$, and the continuity of the norms involved implies Theorem \ref{thm:transition threshold}.

\begin{proof}[Overview of the proof of Theorem \ref{thm:bootstrap}]\label{overview_btrstrap}
The core difficulty of the proof is to obtain sufficiently strong bounds on the various bilinear interactions of the different modes. Roughly speaking, hereby interactions involving two non-zero modes are smallest,\footnote{at least for $\alpha\lesssim 1$ -- our proof shows that for sufficiently large $\alpha\gg 1$ the other interactions can be made as small or even smaller. This shows that for sufficiently large $\alpha$ the best threshold is indeed given by the largest contributions of the self-interactions of nonzero modes.} and in this simplified exposition we shall thus focus on the key step of controlling the contributions of double and simple zero modes. Since self-interactions of the double zero modes are avoided and simple zero modes are subject to dispersive decay, this allows to obtain the thresholds stated in Theorem \ref{thm:transition threshold}. More precisely, the proof of Theorem \ref{thm:bootstrap} follows a hierarchy determined by the improvements due to dispersive effects, which we outline next.

As a first step, in Proposition \ref{prop:nonlinear_dispersive_estimates} of Section \ref{sec:disp} we establish $L^\infty$ amplitude bounds for the simple zero modes $\widetilde{U}_0^j$, $1\leq j\leq 3$: we obtain these by applying the linear dispersive decay estimate \eqref{eq:lin_disp_est} to the Duhamel formulation of the equation satisfied by $\widetilde{U}_0$. For the initial data, this gives the same decay as in the linear estimate \eqref{eq:main_lin_disp}, whereas the nonlinear contributions can be quantified in terms of the size $\eps$ of the initial data and the rotation parameter $\alpha$.

From this we deduce an improved estimate for the double zero modes, see Proposition \ref{prop:doublezero_estimates} in Section \ref{sec:double_zero}, which under the assumptions \eqref{eq:transition_thm} of Theorem \ref{thm:transition threshold} in particular implies a strengthening of \eqref{eq:00-btstrap}. In addition to the dispersive estimates in Proposition \ref{prop:nonlinear_dispersive_estimates}, this relies on the absence of self-interactions of double zero modes in \eqref{eq:doublezero}, as well as their initial smallness (see \eqref{eq:initial_data} and Remark \ref{rem:nonlinear}\eqref{it:00-mean-vanish}).

Due to the oscillatory nature of the coupling between the simple zero modes (see e.g.\ \eqref{eq:simplezero_QK}), energy estimates need to be carried out jointly for the two unknowns $\widetilde{Q}_0,\widetilde{W}_0$ -- see Section \ref{sec:simple_zero}. Together with the aforementioned dispersive bounds in Proposition \ref{prop:nonlinear_dispersive_estimates} and the smallness of the double zero modes established before in Proposition \ref{prop:doublezero_estimates} we then obtain refined bounds for the simple zero modes in Proposition \ref{prop:simple0_finalbound}, which in particular imply an improvement over \eqref{eq:s0-btstrap}.

Finally, sharpened estimates for the nonzero modes are obtained in Proposition \ref{prop:estim_nonzero} of Section \ref{sec:non-zero}, yielding also improved bounds over \eqref{eq:btstrap_Q_neq}, \eqref{eq:btstrap_W_neq}. Hereby the main multiplier $\cA$ (see \eqref{eq:A-prelim} and Section \ref{sec:prelim}) -- applied equally to both $Q,W$ due to the coupled nature of the equations \eqref{eq:perturb_sys_nl} -- plays a crucial role: it combines the multiplier $\m$ already used for the linear analysis of enhanced dissipation and inviscid damping (see Remark \ref{rem:linear}\eqref{it:lin_mixing} and Section \ref{ssec:nonzero_linear}) with two ghost multipliers $\cM=\cM_1\cM_2$ (similar to ones used e.g.\ in \cites{L18,CZDZW24}) and a time weight. The ghost multipliers $\cM_1$ and $\cM_2$ are instrumental in tracking the enhanced dissipation of $u_{\neq}$ and additional inviscid damping decay of $u^2_{\neq}$ in a time integrated fashion (see e.g.\ Lemma \ref{lem:more_enh_dissip_neq}), respectively, while the time weight $\e^{\delta\nu^{\frac13}t}$ gives a simple way of tracking an explicit fixed time rate for the enhanced dissipation. Moreover, $\cA$ satisfies a product estimate (Lemma \ref{lem:m-product}) that is vital in establishing a threshold below $\cO(\nu)$ for certain interactions between nonzero and simple zero modes (see e.g.\ the proof of Lemma \ref{lem:neq_pressure}).

\end{proof}

\subsection{Additional notation}
Aside from the conventions presented above in Section \ref{sec:proofs_setup} and the notation introduced in \eqref{eq:modes_notation} for the double zero, simple zero, and nonzero modes, we collect here some basic notation. We denote the Fourier transform of a function $\varphi$ on $\TT\times\RR\times\TT$ as follows:
with $(k,\eta,l)\in\ZZ\times\RR\times\ZZ$ we let
\begin{equation}
    \widehat \varphi_{k,l}(\eta):=\frac{1}{4\pi}\int_{\TT\times\RR\times\RR}\varphi(x,y,z)\e^{-i(kx+\eta y+lz)}\,\dd x\,\dd y\,\dd z,
\end{equation}
so that the original function $\varphi$ can be recovered via
\begin{equation}
    \varphi(x,y,z)=\sum_{(k,l)\in\ZZ^2}\int_{\RR}\widehat\varphi_{k,l}(\eta)\e^{i(kx+\eta y+lz)}\,\dd\eta.
\end{equation}
We indicate the $L^2$ inner product with $\l\cdot,\cdot\r$ and the $L^2$ norm with $\norm{\cdot}_{L^2}$. 
Additionally, the $H^s$ norm, $s>0$, of a function $\varphi$ is defined as
\begin{equation}
    \norm{\varphi}^2_{H^s}:=\sum_{(k,l)\in\ZZ^2}\int_\RR\l k,\eta,l\r^{2s}|\hat\varphi_{k,l}(\eta)|^2\dd \eta,
\end{equation}
where
\begin{equation}
    \l k,\eta,l\r:=\sqrt{1+|k,\eta,l|^2}, \quad |k,\eta,l|^2:=k^2+\eta^2+l^2.
\end{equation}
For $1\leq p\leq \infty$ and $T>0$, we denote the norm on the space $L^p([0,T];H)$ with $H=L^2$ or $H=H^s$ as
\begin{equation}
    \norm{\varphi}_{L^p([0,T];H)}=\norm{\varphi}_{L^p_tH}.
\end{equation}
Finally, we use $a\lesssim b$ to indicate that $a\leq C b$ for a constant $C>0$ independent of relevant parameters.

\section{Linear analysis}\label{sec:linear}
In this section we consider the linearized dynamics \eqref{eq:linearized} and establish Theorem \ref{thm:linear}. As discussed above in Section \ref{sec:proofs_setup}, for the double zero modes these reduce to the $1d$ heat equation, while the dynamics of the simple zero and nonzero modes can be conveniently described via the unknowns \eqref{eq:qomega} as \eqref{eq:linearsystem-QOmega_0}, which we restate here for ease of reference:
\begin{equation}\label{eq:linearsystem-QOmega}
    \begin{cases}
        \de_t Q+\beta\de_z \Omega^2=\nu\Delta_LQ,\\
        \de_t\Omega^2-(\beta-1)\de_z\Delta_L^{-1}Q=\nu\Delta_L\Omega^2.
    \end{cases}
\end{equation}

\subsection{Dynamics of the simple zero modes}
As is clear from \eqref{eq:linearsystem-QOmega}, depending on $\beta$ we have several different dynamical regimes, which we treat next. Hereby it is particularly instructive to consider the motion of the simple zero modes in \eqref{eq:linearsystem-QOmega}, namely the system
\begin{equation}\label{eq:linearsystem-QOmega_zer0}
    \begin{cases}
        \de_t \widetilde{Q}_0+\beta\de_z \widetilde{\Omega}^2_0=\nu\Delta \widetilde{Q}_0,\\
        \de_t\widetilde{\Omega}^2_0-(\beta-1)\de_z\Delta^{-1}\widetilde{Q}_0=\nu\Delta\widetilde{\Omega}^2_0.
    \end{cases}
\end{equation}
Whether $\beta\in [0,1]$ or not gives completely different behaviors. 

\subsubsection{Instability regime for $0\leq\beta\leq 1$}\label{sec:instable regime}
In this regime we encounter both a secular and an exponential instability:

\medskip

\noindent $\bullet$ $\beta\in\{0,1\}$: \emph{lift-up effect}. 
When $\beta=0$ or $\beta=1$, lift-up occurs. This is classical in the case $\beta=0$ without rotation, but also witnessed here in case the speed of rotation matches the slope of the Couette flow, i.e.\ if $\beta=1$. In the latter case, the equations for the simple zero modes read
    \begin{equation}
     \begin{cases}
        \de_t \widetilde{Q}_0+\de_z \widetilde{\Omega}^2_0=\nu\Delta \widetilde{Q}_0,\\
        \de_t\widetilde{\Omega}^2_0=\nu\Delta\widetilde{\Omega}^2_0,
    \end{cases}   
    \end{equation}
    from which ones sees that $\Omega^2_0$ follows a pure heat equation dynamic, the initial data of which lead to a transient linear growth (for $t\ll \nu^{-1}$) in $\widetilde{Q}_0$: we have that
    \begin{equation}
      \widetilde{\Omega}^2_0(t)=\e^{\nu t\Delta}\widetilde{\Omega}^{2}_0(0),\qquad \widetilde{Q}_0(t)=\e^{\nu t\Delta}\widetilde{Q}_0(0)-t\,\e^{\nu t\Delta}\de_z\widetilde{\Omega}_0^{2}(0).
    \end{equation}
    
\medskip

\noindent $\bullet$ $0<\beta<1$: \emph{exponential instability}. When $0<\beta<1$, the simple zero mode equations \eqref{eq:linearsystem-QOmega_zer0} are diagonalized in the variables 
    \begin{equation}
        D_\pm:=\widetilde{Q}_0\pm \sqrt{\frac{\beta}{1-\beta}}\frac{\de_z}{\abs{\de_z}}\abs{\nabla}\widetilde{\Omega}^2_0,
    \end{equation}
    which satisfy
    \begin{equation}\label{eq:lin_instab}
        \de_t D_{\pm} \pm \sqrt{\beta(1-\beta)}\frac{\abs{\de_z}}{\abs{\nabla}}D_\pm=\nu \Delta D_{\pm}.
    \end{equation}
    \begin{figure}[ht]
    \begin{center}
    \includegraphics[width=7cm]{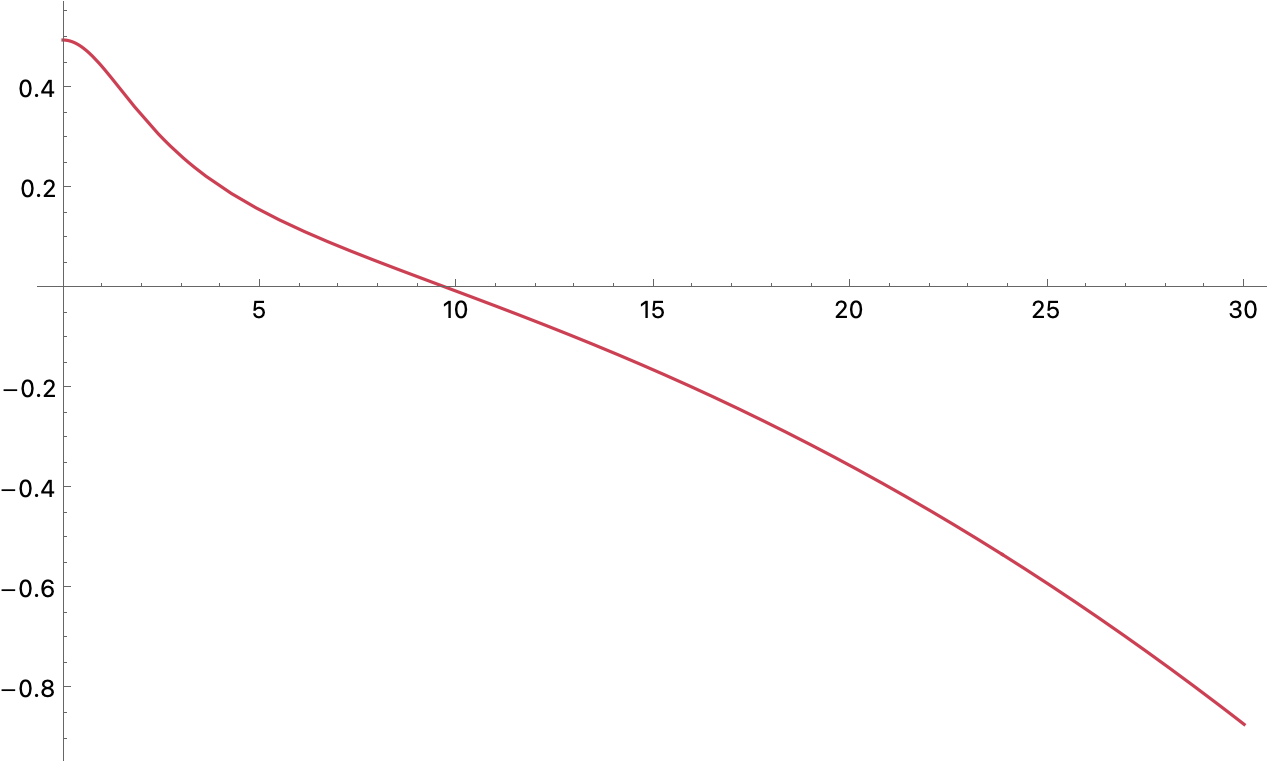}\hfill\includegraphics[width=7cm]{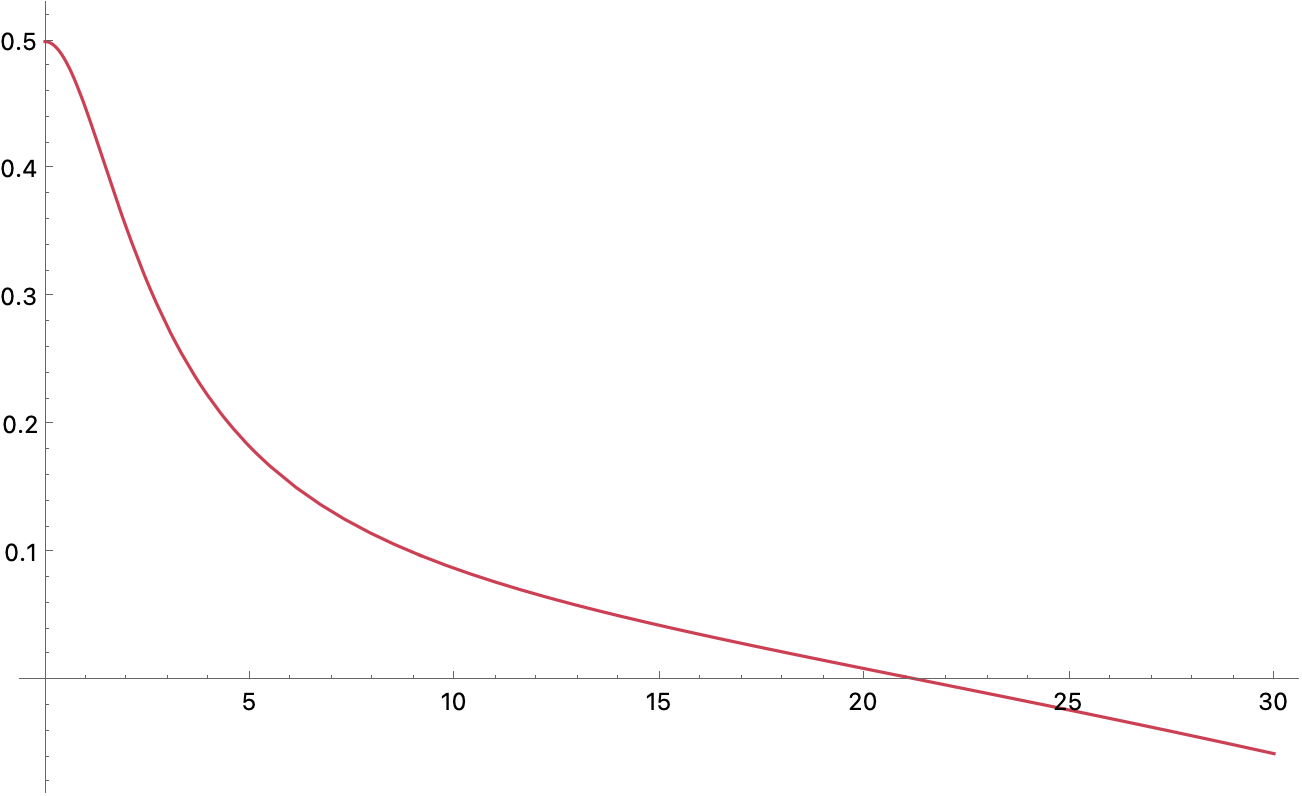}
    \caption{The (unstable) eigenmode $\lambda^+_{\nu,\beta}(\eta,l)=-\nu\abs{(\eta,l)}^2+\sqrt{\beta(1-\beta)}\frac{\abs{l}}{\abs{(\eta,l)}}$ as a function of $\eta\in[0,30]$ for fixed $l=2$, $\beta=0.5$ and $\nu=0.001$ (left) resp.\ $\nu=0.0001$ (right).}\label{fig:ev_plots}
    \end{center}
    \end{figure}
    For $D_-$, this yields the claimed exponential instability provided that $\nu\ll\min\{\sqrt{\beta},\sqrt{1-\beta}\}$. More precisely, upon taking a Fourier transform \eqref{eq:lin_instab} becomes
    \begin{equation}
        \de_t \widehat{D}_\pm(\eta,l,t)=\lambda^\pm_{\nu,\beta}(\eta,l)\widehat{D}_\pm(\eta,l,t),\qquad \lambda^\pm_{\nu,\beta}(\eta,l)=-\nu\abs{(\eta,l)}^2\mp\sqrt{\beta(1-\beta)}\frac{\abs{l}}{\abs{(\eta,l)}},
    \end{equation}
    see also Figure \ref{fig:ev_plots}. In particular, the component $\mathcal{F}^{-1}\left(\widehat{D}_-(0)\mathbbm{1}_{S}\right)$ of the initial data of $D_-(0)$ with frequency support in $S:=\{(\eta,l)\in \RR\times\ZZ:\lambda^-_{\nu,\beta}(\eta,l)>0\}$ undergoes an exponential amplitude inflation.
    To quantify this as in part \eqref{it:linear_unstable} of Theorem \ref{thm:linear}, it suffices to let
    \begin{equation}
        \emptyset\neq S':=\left\{(\eta,l)\in \RR\times\ZZ:\lambda^-_{\nu,\beta}(\eta,l)>\frac12(\sqrt{\beta(1-\beta)}-\nu)\right\}\subset S,
    \end{equation}
    define $\PP$ as $\widehat{\PP \varphi}=\mathbbm{1}_{S'}\widehat{\varphi}$ -- see also Figure \ref{fig:area_plots} -- and unwind the expressions of $\widetilde{u}_0$ in terms of $\widetilde{Q}_0,\widetilde{W}_0$.  
    \begin{figure}[ht]
    \begin{center}
    \includegraphics[width=7cm]{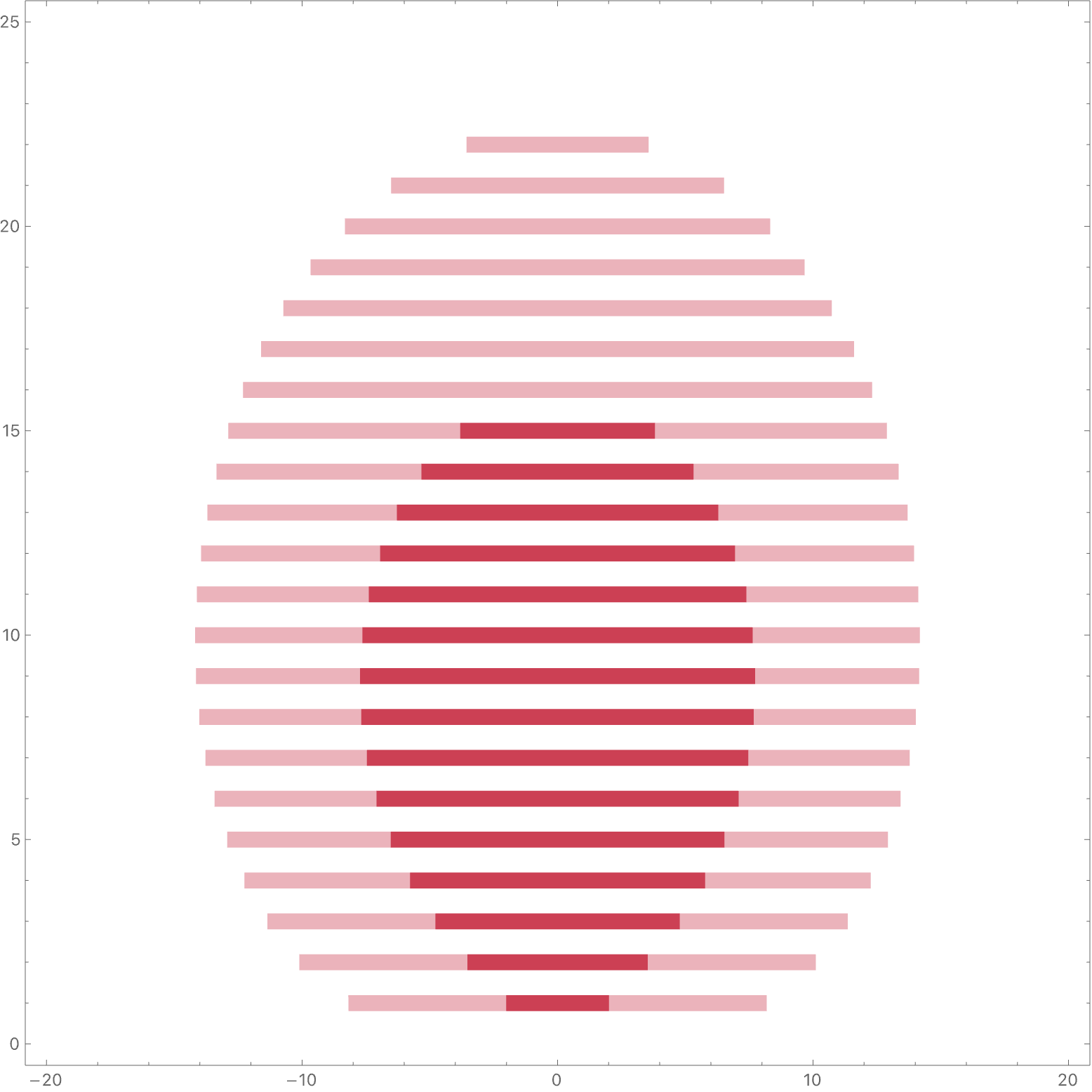}\hfill\includegraphics[width=7cm]{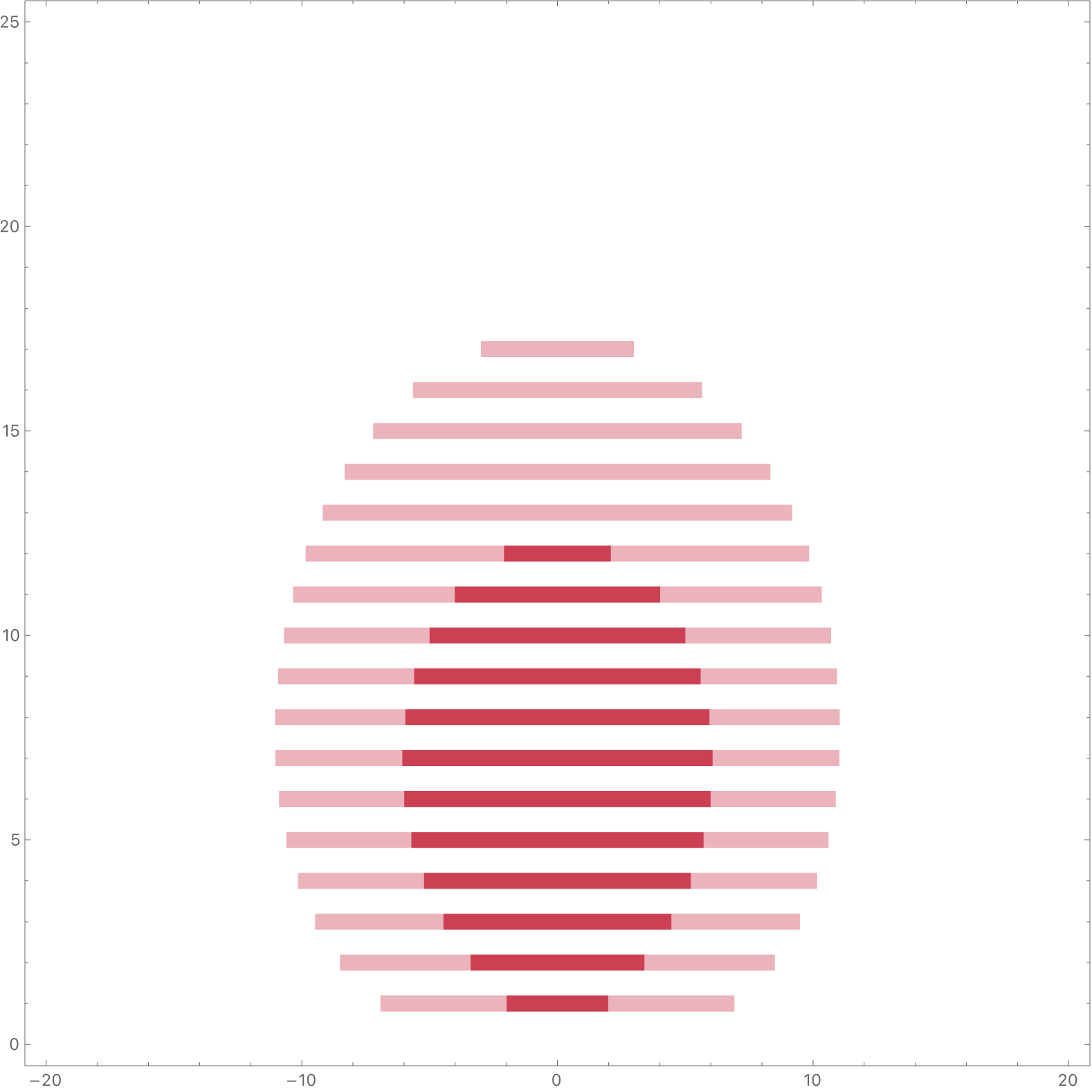}
    \caption{$S'\cap\{l>0\}$ (dark red) inside $S\cap\{l>0\}$ (red) for $\nu=0.001$ and $\beta=0.5$ (left) resp.\ $\beta=0.1$ (right).}\label{fig:area_plots}
    \end{center}
    \end{figure} 
This proves part \eqref{it:linear_unstable} of Theorem \ref{thm:linear}.

\subsubsection{Stable regime $B_\beta>0$: $\beta<0$ or $\beta>1$}\label{ssec:lin_stable}
In this regime, the coupling (due to rotation) of $Q$ and $\Omega$ is of oscillatory nature. To track the corresponding dynamics, it is convenient to rescale the variable $\Omega^2$ to match the regularity of $Q$:
\begin{lemma}
    Let $\beta<0$ or $\beta>1$. The system \eqref{eq:linearized} can equivalently be expressed via the variables
    \begin{equation}\label{eq:unknows}
        Q:=\Delta_L U^2,\qquad W:=\sqrt{\frac{\beta}{\beta-1}}|\nabla_L|\Omega^2,
    \end{equation}
    in the moving frame as
\begin{equation}\label{eq:linearsystem-QK}
    \begin{cases}
        \de_t Q+\alpha \de_z|\nabla_L|^{-1} W=\nu\Delta_LQ,\\
        \de_tW -\de_x\de_y^L|\nabla_L|^{-2} W+\alpha \de_z|\nabla_L|^{-1}Q=\nu\Delta_L W,
    \end{cases}
    \end{equation}
where $\alpha=\sqrt{\beta(\beta-1)}>0$,
together with the equations \eqref{eq:doublezero} for the double-zero modes.
\end{lemma}

For the simple zero modes, the purely oscillatory nature of the coupling can be easily seen in \eqref{eq:linearsystem-QOmega_zer0} resp.\ \eqref{eq:linearsystem-QK}.
\begin{proof}[Proof of part (\ref{it:linear_stable}) Theorem \ref{thm:linear} -- simple zero modes]
With the change of coordinates \eqref{eq:unknows} the simple zero modes system \eqref{eq:linearsystem-QOmega_zer0} reads
\begin{equation}\label{eq:linearsystem-QK-simplezero}
    \begin{cases}
        \de_t \widetilde{Q}_{0}+\alpha \de_z|\nabla|^{-1} \widetilde{W}_{0}=\nu\Delta \widetilde{Q}_{0},\\
        \de_t\widetilde{W}_{0} +\alpha \de_z|\nabla|^{-1} \widetilde{Q}_{0}=\nu\Delta \widetilde{W}_{0}.
    \end{cases}
\end{equation}
which is diagonalized by the unknowns
\begin{equation}\label{eq:Gpm}
G_\pm:=\widetilde{Q}_0\pm i\de_z\abs{\de_z}^{-1}\widetilde{W}_0
\end{equation}
satisfying 
\begin{equation}\label{eq:Gpm-Lpm}
    \de_t G_\pm =\cL_\pm G_\pm, \qquad 
    \cL_\pm:= \nu\Delta\pm i \alpha |\de_z||\nabla|^{-1}.
\end{equation}
The operator $\cL_\pm$ encodes the heat dissipation from $\nu\Delta$, and the oscillatory behavior, $i\alpha|\de_z||\nabla|^{-1}$, arising from the structure of the system \eqref{eq:linearsystem-QK-simplezero}. Remarkably, this is the same operator appearing in the simple zero modes analysis of the $3d$ Boussinesq system when perturbed around a stably stratified Couette flow, see \cites{CZDZ23,CZDZW24}. It follows that
\begin{equation}
    \norm{G_\pm(t)}^2_{L^2}=\norm{\e^{\cL_\pm t}G_\pm(0)}^2_{L^2}=\sum_{l\neq0}\int_{\RR} \left|\e^{-\nu(\eta^2+l^2)\pm i\alpha |l||\eta,l|^{-1}}\widehat G_\pm(0,\eta,l)\right|^2\dd \eta \leq \e^{-2\nu t}\norm{G_\pm(0)}_{L^2}^2,
\end{equation}
and hence can deduce the suppression of the lift up effect
\begin{equation}
    \norm{\widetilde u^j_0(t)}_{L^2}\leq 2\norm{ G_\pm(t)}_{L^2}\leq 2\e^{-\nu t}\norm{G_\pm(0)}_{L^2}\leq  4\e^{-\nu t}\norm{\widetilde{u}_0(0)}_{H^2}, \qquad j\in\{1,2,3\}.
\end{equation}
In addition, as in \cite{CZDZW24}*{Corollary 3.2}, we have the linear dispersive estimates for $\cL_\pm$
\begin{equation}\label{eq:lin_disp_est}
    \sum_{\substack{0\leq a_j\leq 2\\1\leq j\leq 4}}\|(\de_z\abs{\de_z}^{-1})^{a_1}(\de_y\abs{\nabla}^{-1})^{a_2}\abs{\de_z}^{-a_3}\abs{\nabla}^{-a_4}\e^{t\cL_\pm}\widetilde{\varphi}_0\|_{L^\infty(\RR\times\TT)}\leq C (\alpha t)^{-\frac13}\e^{-\nu t}\|\widetilde{\varphi}_0\|_{W^{4,1}
    (\RR\times\TT)}.
\end{equation}

Since by \eqref{eq:u_recover} and \eqref{eq:unknows} we have that
\begin{equation}
 \widetilde{U}^1_0=-\sqrt{\frac{\beta-1}{\beta}}\de_z\abs{\de_z}^{-2}\abs{\nabla}^{-1}\widetilde{W}_0,\qquad \widetilde{U}^2_0=-\abs{\nabla}^{-2}\widetilde{Q}_0,\qquad \widetilde{U}^3_0=-\de_z\abs{\de_z}^{-2}\de_y\abs{\nabla}^{-2}\widetilde{Q}_0,   
\end{equation}
this implies the claimed decay for all components $\widetilde{u}^j_0=\widetilde{U}^j_0$, $1\leq j\leq 3$, of the velocity field.

 \end{proof}

\subsection{Dynamics of the nonzero modes}\label{ssec:nonzero_linear}
In contrast to the simple zero modes, the dynamics of non-zero modes are dominated by shearing due to the Couette background. As observed in \cite{BGM17}, to capture the associated effects of inviscid damping and enhanced dissipation it is convenient to use a suitably constructed Fourier multiplier $\m$ with symbol $m=m(t,k,\eta,l)$, 
defined as $m(0,k,\eta,l)=1$ and
\begin{equation}\label{def:mult-m}
    \frac{\dot m}{m}=\begin{cases}
        \frac{k(\eta-kt)}{|k,\eta-kt,l|^2}, \qquad t\in [\frac{\eta}{k},\frac{\eta}{k}+1000\nu^{-\frac13}],\\
        0, \hspace*{2.065cm} t\notin [\frac{\eta}{k},\frac{\eta}{k}+1000\nu^{-\frac13}].
    \end{cases}
\end{equation}
The central role of $\m$ is to deal with the slowly decaying term $\de_{x}\de_y^L\abs{\nabla_L}^{-2}W$ in \eqref{eq:linearsystem-QK}, balancing the growth that solutions experience in the transition between the inviscid damping regime $t< \frac{\eta}{k}$, and the dissipative regime $t\gtrsim\frac{\eta}{k}+\nu^{-\frac13}$. We recall from \cite{BGM17}*{Section 2.3} the main properties.
\begin{lemma}[Properties of $\fm$ ]\label{lemma:prop-m}
For $t\geq 0$ and $(k,\eta,l)\in(\TT\times\RR\times\TT)\setminus\{0\}$ we have:
    \begin{enumerate}
        \item (Uniform bounds) There holds that
        \begin{equation}\label{eq:lowerbound-m}
            \nu^{\frac13}+\frac{|k,l|}{|k,\eta-kt,l|}\lesssim m \lesssim 1.
        \end{equation}
        \item (Product estimate) For any $(\eta',l')\in\RR\times\TT$ there holds that
        \begin{equation}\label{eq:commutator-m-zeromodes}
            m(t,k,\eta,l)\lesssim m(t,k,\eta',l') \; \l\eta-\eta',l-l'\r.
        \end{equation}
    \end{enumerate}
\end{lemma}
The estimates \eqref{eq:lowerbound-m} play a key role in the quantitative analysis of stability of non-zero modes in \eqref{eq:linearsystem-QK}, which we present next. The bound \eqref{eq:commutator-m-zeromodes} will be useful for the related nonlinear analysis -- see Lemma \ref{lem:m-product} and Section \ref{sec:non-zero}.

\begin{proof}[Proof of part (\ref{it:linear_stable}) Theorem \ref{thm:linear} -- non-zero modes]
Rewriting the nonzero modes of the system \eqref{eq:linearsystem-QK} in Fourier variables $(k,\eta,l)\in\ZZ\setminus\{0\}\times\RR\times\ZZ$ we obtain 
\begin{align}
        \de_t \widehat Q_{\neq}+i\alpha \frac{l}{|k,\eta-kt,l|}\widehat W_{\neq}&=-\nu|k,\eta-kt,l|^2\widehat Q_{\neq},\\
        \de_t\widehat W_{\neq} + \frac{k(\eta-kt)}{|k,\eta-kt,l|^{2}}\widehat W_{\neq}+i\alpha \frac{l}{|k,\eta-kt,l|}\widehat Q_{\neq}&=-\nu|k,\eta-kt,l|^2 \widehat W_{\neq}.
\end{align}
The presence of the extra term $\frac{k(\eta-kt)}{|k,\eta-kt,l|^{2}}\widehat W_{\neq}$ motivates the definition of the multiplier $m$, which in turn is included in the following point-wise energy estimate in both variables to maintain the cancellation of skew-symmetric terms. 
We have
\begin{equation}\label{eq:nonzero_linear_energy}
\begin{aligned}
    &\frac12\ddt\left(\abs{m\widehat Q_{\neq}}^2+\abs{m\widehat W_{\neq}}^2\right) + \abs{\sqrt{-\frac{\dot m}{m}}m\widehat Q_{\neq}}^2+\abs{\sqrt{-\frac{\dot m}{m}}m\widehat W_{\neq}}^2 +\nu|k,\eta-kt,l|^2\left(\abs{m\widehat Q_{\neq}}^2+\abs{m\widehat{W}_{\neq}}^2\right)\\
    &\quad= -\frac{k(\eta-kt)}{|k,\eta-kt,l|^2}\abs{m\widehat W_{\neq}}^2\chi_{(0,\frac{\eta}{k}]}(t) - \frac{\dot m}{m}\abs{m\widehat W_{\neq}}^2\chi_{[\frac{\eta}{k},\frac{\eta}{k}+1000\nu^{-\frac13}]}(t) \\
    &\qquad -\frac{k(\eta-kt)}{|k,\eta-kt,l|^2}\abs{m\widehat W_{\neq}}^2\chi_{[\frac{\eta}{k}+1000\nu^{-\frac13},\infty)}(t)\\
    &\quad \leq \abs{\sqrt{-\frac{\dot m}{m}}m\widehat W_{\neq}}^2+ \frac{\nu}{2}|k,\eta-kt,l|^2\abs{m\widehat W_{\neq}}^2.
\end{aligned}
\end{equation}
The final bound follows from noticing that for $k\neq 0$ and $0<t\leq \frac{\eta}{k}$, the term $\frac{k(\eta-kt)}{|k,\eta-kt,l|^2}$ is non-negative, and that for $t\geq\frac{\eta}{k}+1000\nu^{-\frac13}$ it holds 
\begin{equation}
    \frac{|k(\eta-kt)|}{|k,\eta-kt,l|^2}\leq \frac{\nu}{2}|k,\eta-kt,l|^2.
\end{equation}
We deduce 
\begin{equation}
    \ddt\left(\abs{m\widehat Q_{\neq}}^2+\abs{m\widehat W_{\neq}}^2\right) +\nu|k,\eta-kt,l|^2\left(\abs{m\widehat Q_{\neq}}^2+\abs{m\widehat{W}_{\neq}}^2\right)\leq 0,
\end{equation}
and hence integrating in time and using 
\begin{equation}
    \int_0^t|k,\eta-ks,l|^2 \dd s\geq \frac{1}{12}k^2t^3,
\end{equation}
leads to
\begin{equation}
    \abs{m(t)\widehat Q_{\neq}(t)}^2+\abs{m(t)\widehat W_{\neq}(t)}^2 \leq \e^{-\frac{1}{12}\nu t^3}\left(\abs{\widehat Q_{\neq}(0)}^2+\abs{\widehat W_{\neq}(0)}^2\right),
\end{equation}
which implies the desired $L^2$-based norm bound.
To conclude, we translate to the original velocity variable $U$ the above bounds. From \eqref{eq:u_recover} 
it follows
\begin{equation}
    \abs{\widehat U^2_{\neq}}=\abs{|k,\eta-kt,l|^{-2}|\widehat Q_{\neq}}\leq \abs{|k,\eta-kt,l|^{-1}m\widehat Q_{\neq}}
\end{equation}
and hence
\begin{align}
    \norm{\widehat{u}^2(k,\cdot,l,t)}_{L^2}^2&=\norm{\widehat{U}^2(k,\cdot,l,t)}_{L^2}^2=\int_\RR |k,\eta-kt,l|^{-2}\abs{m\widehat Q(k,\eta,l,t)}^2\dd \eta\notag\\
    &\leq \e^{-\frac{1}{12}\nu t^3}\int_\RR |k,\eta-kt,l|^{-2}\abs{\widehat Q(k,\eta,l,0)}\dd \eta\notag\\
    &\leq \l t\r^{-2}\e^{-\frac{1}{12}\nu t^3}\int_\RR |k,\eta,l|^{6}\abs{\widehat U^2(k,\eta,l,0)}\dd \eta,
\end{align}
which proves the $u^2_{\neq}$ estimate in Theorem \ref{thm:linear} upon summation over $k,l$, where $k\neq 0$, and using Plancherel.

Analogously for $U^1, U^3$ we appeal to \eqref{eq:u_recover} and obtain the point-wise estimates
\begin{align}
    \abs{\widehat U^1_{\neq}}&\leq \sqrt{\frac{\beta-1}{\beta}}\abs{l|k,l|^{-2}|k,\eta-kt,l|^{-1}\widehat W_{\neq}}+\abs{k(\eta-kt)|k,\eta-kt,l|^{-2}|k,l|^{-2}\widehat Q_{\neq}}\notag\\
    &\leq \sqrt{\frac{\beta-1}{\beta}}\abs{m\widehat W_{\neq}}+\abs{m\widehat Q_{\neq}}\\
    &\leq \left(\sqrt{\frac{\beta-1}{\beta}} +1\right)\e^{-\frac{1}{12}\nu t^3} \left(\abs{\widehat W_{\neq}(0)}+\abs{\widehat Q_{\neq}(0)}\right),
\end{align}
and 
\begin{align}
    \abs{\widehat U^3_{\neq}}&\leq \sqrt{\frac{\beta-1}{\beta}}\abs{k|k,l|^{-2}|k,\eta-kt,l|^{-1}\widehat W_{\neq}}+\abs{l(\eta-kt)|k,\eta-kt,l|^{-2}|k,l|^{-2}\widehat Q_{\neq}}\notag\\
    &\leq \sqrt{\frac{\beta-1}{\beta}}\abs{m\widehat W_{\neq}}+\abs{m\widehat Q_{\neq}}\\
    &\leq \left(\sqrt{\frac{\beta-1}{\beta}} +1\right)\e^{-\frac{1}{12}\nu t^3} \left(\abs{\widehat W_{\neq}(0)}+\abs{\widehat Q_{\neq}(0)}\right).
\end{align}
This leads to
\begin{align}
    \norm{u^1_{\neq}(t)}^2_{L^2}+ \norm{u^3_{\neq}(t)}^2_{L^2}
    &\lesssim \left(1+\frac{\beta-1}{\beta}\right)\left( \norm{\fm W_{\neq}(t)}_{L^2}^2+\norm{\fm Q_{\neq}(t)}_{L^2}^2\right)\\
    &\lesssim  \left(1+\frac{\beta-1}{\beta}+\frac{\beta}{\beta-1}\right) \e^{-\frac{1}{12}\nu t^3}\norm{u_{\neq}(0)}_{H^2}^2 ,
\end{align}
and concludes the proof.
\end{proof}

\section{Nonlinear analysis}\label{sec:nonlinear}
This section is devoted to the proof of Theorem \ref{thm:bootstrap}. After some technical preliminaries in Section \ref{sec:prelim} and a short recap of the basic dispersive features of our system in Section \ref{sec:disp}, we establish the energy estimates for the double zero modes in Section \ref{sec:double_zero}, for the simple zero modes in Section \ref{sec:simple_zero} and for the non-zero modes in Section \ref{sec:non-zero}. Throughout this section we will make generous use of the bootstrap assumptions of Theorem \ref{thm:bootstrap} when relevant, without explicitly referring to them every time.

To begin, we recall from \eqref{eq:perturb_sys_nl} the nonlinear evolution equations for our unknowns $W,Q$ in \eqref{eq:unknows} is
\begin{equation}\label{eq:nonlin}
    \begin{cases}
        \de_t Q +\alpha \de_z|\nabla_L|^{-1}W -\nu \Delta_L Q = -\de_y^L\Delta_L\mathcal{P}(U,U)-\Delta_L\mathcal{T}(U,U^2),\\
        \de_t W-\de_{x}\de_y^L\abs{\nabla_L}^{-2}W +\alpha \de_z|\nabla_L|^{-1}Q-\nu\Delta_L W=\sqrt{\frac{\beta}{\beta-1}}\abs{\nabla_L}\left(\de_z\mathcal{T}(U,U^1)-\de_x\mathcal{T}(U,U^3)\right),
    \end{cases}
\end{equation}
where the transport and pressure nonlinearities are given by
\begin{equation}
 \mathcal{T}(U,U^j)=U\cdot\nabla_LU^j,\quad 1\leq j\leq 3,\qquad \mathcal{P}(U,U)=-\Delta_L^{-1}\sum_{1\leq i,j\leq 3}\partial_i^LU^j\partial_j^LU^i.   
\end{equation} 
The system \eqref{eq:nonlin} is to be complemented with the double-zero mode equations \eqref{eq:doublezero_eq} for $\overline{U}^j_0$, $j=1,3$ (recall that $\overline{U}^2_0=0$ by incompressibility and integrability) -- the simple zero and non-zero modes can be recovered from $Q,W$ as (see \eqref{eq:u_recover} and \eqref{eq:unknows})
\begin{equation}\label{eq:U-QW}
\begin{aligned}
 U^1-\overline{U}^1_0&=-\abs{\nabla_{x,z}}^{-2}\left(\sqrt{\frac{\beta-1}{\beta}}\de_z\abs{\nabla_L}^{-1}W+\de_x\de_y^L\abs{\nabla_L}^{-2}Q\right),\\
 U^2&=-\abs{\nabla_L}^{-2}Q,\\
 U^3-\overline{U}^3_0&=-\abs{\nabla_{x,z}}^{-2}\left(\sqrt{\frac{\beta-1}{\beta}}\de_x\abs{\nabla_L}^{-1}W-\de_z\de_y^L\abs{\nabla_L}^{-2}Q\right).
\end{aligned} 
\end{equation}
As is apparent e.g.\ from \eqref{eq:U-QW} and stated in Theorem \ref{thm:transition threshold}, many of our bounds for nonlinear terms degenerate as $B_\beta\searrow 0$. Since we are not interested in this regime but instead in the effect of large $\abs{\beta}$, in the remainder of this section we will track dependence of constants on $\beta$ only to the extent that they are relevant for large Bradshaw-Richardson numbers, i.e.\ we will ignore constants that are uniformly bounded as $\abs{\beta}\to\infty$.

\subsection{Preliminaries}\label{sec:prelim}

To control the evolution of the nonzero modes under the nonlinear system \eqref{eq:nonlin}, we introduce the Fourier multiplier
\begin{equation}\label{eq:defA}
    \cA:=\fm\cM \e^{\delta \nu^{\frac13}t},
\end{equation}
which incorporates the multiplier $\fm$ already used in the linear analysis (see Section \ref{sec:linear}, \eqref{def:mult-m} and Lemma \ref{lemma:prop-m}) and additional ghost weights $\cM:=\cM_1\cM_2$ as well as the time weight $\e^{\delta\nu^{\frac13}t}$, where
\begin{itemize}
    \item $\cM_1$ is defined via its Fourier symbol $M_1(t,k,\eta,l)$ as
    \begin{equation}
    -\frac{\dot M_1(t,k,\eta,l)}{M_1(t,k,\eta,l)}:=\frac{\nu^{\frac13}k^2}{k^2+\nu^{\frac23}(\eta-kt)^2}, \quad M_1(0,k,\eta,l)=1,
    \end{equation}
    and satisfies moreover that $0<c\leq M_1\leq 1$ for some $c>0$. This multiplier has already been used in \cites{BGM17,L18,HSX24sept,CZDZW24} and captures the enhanced dissipation effect due to the background Couette flow through the following key property:
\begin{lemma}
    There exists $c>0$ such that
    \begin{equation}
        0<c\leq M_1\leq 1.
    \end{equation}
    Moreover, there holds that
    \begin{equation}\label{def:multi-M1}
        \nu^{1/3}\norm{F_{\neq}}^2_{L^2_tH^N}\lesssim \norm{\sqrt{-\frac{\dot \cM_1}{\cM_1}}F_{\neq}}^2_{L^2_tH^N}+\nu\norm{\nabla_L F_{\neq}}^2_{L^2_tH^N}.
    \end{equation}
\end{lemma}
The proof of this lemma can be found in \cite{BGM17}*{Lemma 2.1} (or see \cite{CZDZW24}*{Lemma 2.2}).

\item $\cM_2$ is designed to capture additional time decay for some terms involving $U^2_{\neq}$ (in a similar spirit as in our nonlinear analysis in \cite{CZDZW24}), and is defined via the Fourier multiplier
\begin{equation}\label{def:multi-M2}
    -\frac{\dot M_2(t,k,\eta,l)}{M_2(t,k,\eta,l)}:=\frac{k^2}{k^2+(\eta-kt)^2+l^2}, \qquad M_2(0,k,\eta,l)=0.
\end{equation}
It also satisfies $0<c\leq M_2\leq 1$, for some $c>0$ (see e.g.\ \cite{CZDZW24}*{Lemma 2.1}), and is used through the combination of the energy estimates \eqref{eq:btstrap_Q_neq} and \eqref{eq:extratimedecay}.

\item the time weight $\e^{\delta\nu^{\frac13}t}$ provides a simple way of tracking the enhanced dissipation rate, where $\delta>0$ will be chosen sufficiently small (see the proof of Proposition \ref{prop:estim_nonzero}).
\end{itemize}

We record that the main multiplier $\cA$ satisfies
\begin{equation}\label{eq:A-Sobbd}
\norm{\cA}_{H^s\to H^s}\lesssim \e^{\delta \nu^{\frac13}t},\qquad s\geq 0,
\end{equation}
as can be seen directly from \eqref{eq:lowerbound-m}, and we have the following bounds:
\begin{lemma}\label{lem:more_enh_dissip_neq}
 There exists $C>0$ such that under the bootstrap assumptions of Theorem \ref{thm:bootstrap} there holds that
 \begin{equation}\label{eq:more_enh_dissip_neq}
  \sum_{F\in\{Q,W\}}\nu^{1/3}\norm{\cA F_{\neq}}^2_{L^2_tH^N}\leq C\eps^2.
 \end{equation}
\end{lemma}
\begin{proof}
 This follows directly by combining \eqref{def:multi-M1} with the bootstrap assumptions \eqref{eq:btstrap_Q_neq} and \eqref{eq:btstrap_W_neq}.
\end{proof}
Moreover, the following (nonlinear) product estimate will be useful in Section \ref{sec:non-zero}.
\begin{lemma}\label{lem:m-product}
 For $F,G:\TT\times\RR\times\TT\to\RR$  and $s>\frac32$ there holds that
 \begin{equation}\label{eq:A-product-bd}
     \norm{\cA(FG_0)}_{H^s}\lesssim \norm{\cA F}_{H^s}\norm{\l\nabla\r G_0}_{H^s},
 \end{equation}
\end{lemma}
\begin{proof}
 From \eqref{eq:commutator-m-zeromodes} and the Sobolev algebra property of $H^s$, $s>\frac32$, we directly obtain that
 \begin{equation}\label{eq:m-product-bd}
     \norm{\fm(FG_0)}_{H^s}\lesssim \norm{\fm F}_{H^s}\norm{\l\nabla\r G_0}_{H^s}.
 \end{equation}
 The claim \eqref{eq:A-product-bd} then follows by invoking the ghost multiplier bounds $0<c\leq M_j\leq 1$, $j\in\{1,2\}$.
\end{proof}

For future reference we also collect some basic consequences of our choice of  good unknowns \eqref{eq:unknows}  and the bootstrap setup.
\begin{lemma}\label{lemma:additional_estim_U}
    For the velocity $U$ we have the following estimates, valid for any $s\geq 0$:
    \begin{itemize}
    \item Nonzero modes - absorption of $|\nabla_L|$:
    \begin{equation}\label{eq:absorption_nablaL}
        \norm{|\nabla_{x,z}||\nabla_L|U^{1,3}_{\neq}}_{H^s}\lesssim  \nu^{-\frac13}\norm{\mathfrak{m}(Q,W)_{\neq}}_{H^s},\qquad
        \norm{|\nabla_L|^2U^2_{\neq}}_{H^s}\lesssim  \nu^{-\frac13}\norm{\fm Q_{\neq}}_{H^s}.
    \end{equation}
    \item Nonzero modes - reconstruction of $\fm$:
    \begin{equation}\label{eq:reconstruction_m}
        \norm{|\nabla_{x,z}|^2U^{1,3}_{\neq}}_{H^s}\lesssim  \norm{\fm (Q,W)_{\neq}}_{H^s},\qquad \norm{|\nabla_{x,z}||\nabla_L|U^2_{\neq}}_{H^s}\lesssim  \norm{\fm Q_{\neq}}_{H^s}.
    \end{equation}
    \item Nonzero modes - extra time decay of $U^2_{\neq}$:
    \begin{align}\label{eq:extratimedecay}
        \norm{|\nabla_{x,z}|U^2_{\neq}}_{H^s}\lesssim  \norm{\sqrt{-\frac{\dot \cM_2}{\cM_2}}\fm Q_{\neq}}_{H^s}.
    \end{align}
    \item Simple zero modes - conversion to good unknowns:
    \begin{equation}\label{eq:simple0_conversion}
        \norm{|\de_z||\nabla|\widetilde{U}^1_0}_{H^s}\lesssim  \norm{\widetilde{W}_0}_{H^s},\quad
        \norm{|\nabla|^{2}\widetilde{U}^2_{0}}_{H^s}\lesssim   \norm{\widetilde{Q}_0}_{H^s},\quad
        \norm{|\de_z||\nabla|\widetilde{U}^3_0}_{H^s}\lesssim   \norm{\widetilde{Q}_0}_{H^s}.
    \end{equation}
    \end{itemize}
\end{lemma}
\begin{proof}
    The proof follows from \eqref{eq:u_recover} and Lemma \ref{lemma:prop-m}. 
\end{proof}

\subsection{Dispersive decay of simple zero modes}\label{sec:disp}
A key element in our analysis is the behavior of simple zero modes, which by \eqref{eq:nonlin} satisfy
\begin{equation}\label{eq:simplezero_QK}
    \begin{cases}
        \de_t \widetilde{Q}_0 +\alpha \de_z|\nabla|^{-1}\widetilde{W}_0 -\nu \Delta \widetilde{Q}_0 = -\de_y\Delta\widetilde{\mathcal{P}}_0(U,U)-\Delta\widetilde{\mathcal{T}}_0(U,U^2),\\
        \de_t \widetilde{W}_0 +\alpha \de_z|\nabla|^{-1}\widetilde{Q}_0-\nu\Delta \widetilde{W}_0=\sqrt{\frac{\beta}{\beta-1}}\abs{\nabla}\de_z\widetilde{\mathcal{T}}_0(U,U^1).
    \end{cases}
\end{equation}

Having studied the linearization of \eqref{eq:simplezero_QK} in Section \ref{ssec:lin_stable} and in particular established the dispersive decay \eqref{eq:main_lin_disp}, we can build on this to obtain amplitude bounds also for the nonlinear evolution of the simple zero modes: 
\begin{proposition}[Compare Proposition 3.3 in \cite{CZDZW24}]\label{prop:nonlinear_dispersive_estimates}
Under the assumptions of Theorem \ref{thm:bootstrap}, there holds that
    \begin{equation}\label{eq:nonlinear_dispersive_estimate}
        \sum_{a\in\NN^2_0, |a|\leq 2}\norm{\nabla_{y,z}^a\widetilde{U}_0(t)}_{L^\infty}\lesssim \alpha^{-\frac13}t^{-\frac13}\e^{-\nu t} \eps + \alpha^{-\frac13}\nu^{-\frac23}\eps^2.
    \end{equation}
\end{proposition}
\begin{proof}
    The proof proceeds as in \cite{CZDZW24}*{Proposition 3.3}, so we refer the reader there for full details and only provide here a sketch of the main steps. To establish the claim, it is sufficient to prove suitable bounds for the complex unknown
    \begin{equation}
        \Gamma := -\abs{\nabla}^{-2}(\widetilde{Q}_0+i\de_z\abs{\de_z}^{-1}\widetilde{W}_0)=\widetilde{U}^2_0-i\sqrt{\frac{\beta}{\beta-1}}|\de_z||\nabla|^{-1}\widetilde{U}^1_0,
    \end{equation}
    from which we can also recover $\widetilde{U}^3_0$ using incompressibility. We note -- see also \eqref{eq:Gpm}, \eqref{eq:Gpm-Lpm} -- that $\Gamma$ satisfies
    \begin{equation}
        \de_t\Gamma=\cL_+\Gamma+\cN(U,\Gamma),\qquad \cL_+=\nu\Delta\pm i \alpha |\de_z||\nabla|^{-1},
    \end{equation}
    where $\cN$ is given by 
    \begin{equation}
    \begin{aligned}
      \cN(U,\Gamma)&=\de_y\widetilde{\mathcal{P}}(U,U)_0-\widetilde{\mathcal{T}}(U,U^2)_0+i\abs{\de_z}\abs{\nabla}^{-1}\widetilde{\mathcal{T}}(U,U^1)_0.
    \end{aligned}  
    \end{equation}
    To prove the claim we write Duhamel's formula
    \begin{equation}
     \Gamma(t)=\e^{t\cL_+}\Gamma(0) +\int_0^t \e^{(t-s)\cL_+}\cN(U,\Gamma)(s)\dd s  
    \end{equation}
    and apply the linear decay estimate \eqref{eq:lin_disp_est} for the semigroup $\cL_+$ (see also \cite{CZDZW24}*{Corollary 3.2}).
\end{proof}

\subsection{Double zero modes}\label{sec:double_zero}
The only double zero modes that have to be studied are $\overline{U}^1_0$ and $\overline{U}^3_0$, since $\overline{U}^2_0=0$ due to incompressibility and integrability. Both satisfy the same equation, which reads
\begin{equation}\label{eq:doublezero_eq}
    \de_t\overline{F}_0 -\nu\de_y^2\overline{F}_0=-\de_y\overline{(\widetilde{U}^2_0\widetilde{F}_0)}_0-\de_y\overline{(U^2_{\neq} F_{\neq})}_0,\qquad F\in\{U^1,U^3\}.
\end{equation}
Remarkably and crucially, there are no double zero self interactions. We exploit Proposition \ref{prop:nonlinear_dispersive_estimates} to prove the following bounds for the $H^{N}$ norm of $\overline{F}_0$, which proves part \ref{it:00-btstrap} of Theorem \ref{thm:bootstrap}.

\begin{proposition}[In analogy to Lemma 3.6 in \cite{CZDZW24}]\label{prop:doublezero_estimates}
    Let $\overline{F}_0$ solve the equation \eqref{eq:doublezero_eq} with $\overline{F}_0(0)=0$ and $F\in\{U^1,U^3\}$, then we have 
    \begin{equation}
        \norm{\overline{F}_0}^2_{L^\infty_tH^{N}}+\nu\norm{\de_y\overline{F}_0}^2_{L^2_tH^{N}}\lesssim (\nu^{-\frac12}\eps+\alpha^{-\frac13}\nu^{-\frac53}\eps^2)^2\eps^2.
    \end{equation}
\end{proposition}

\begin{proof}
 Since $\overline{U}^2_0=0$, there holds that
 \begin{equation}\label{eq:noselfdouble0}
  (\overline{U^2F})_0=(\overline{U^2_0\widetilde{F}_0})+(\overline{U^2_{\neq} F_{\neq}})_0,
 \end{equation}
 and by energy estimates it follows that for $F\in\{U^1,U^3\}$
 \begin{equation}\label{eq:energy_doublezeros}
   \norm{\overline{F}_0(t)}_{H^{N}}^{2}+\nu\int_0^t\norm{\de_y\overline{F}_0}_{H^{N}}^2\dd \tau \leq \nu^{-1}\int_0^t  \norm{(\overline{\widetilde{U}^2_0\widetilde{F}_0})(\tau)}_{H^{N}}^2 +\norm{(\overline{U^2_{\neq} F_{\neq}})_0(\tau)}_{H^{N}}^2\dd\tau.
 \end{equation}
 Since
 \begin{equation}
   \norm{(\overline{\widetilde{U}^2_0\widetilde{F}_0})(\tau)}_{H^{N}}\lesssim \norm{\widetilde{U}_0^2(\tau)}_{L^\infty}\norm{ \widetilde{F}_0(\tau)}_{H^{N}}+\norm{\widetilde{Q}_0(\tau)}_{H^{N}}\norm{\widetilde{F}_0(\tau)}_{L^\infty},
 \end{equation}
 we obtain from Proposition \ref{prop:nonlinear_dispersive_estimates} that
\begin{equation}
 \norm{(\overline{\widetilde{U}^2_0\widetilde{U}_0^1})(\tau)}_{H^{N}}+\norm{(\overline{\widetilde{U}^2_0\widetilde{U^3}_0})(\tau)}_{H^{N}}\lesssim \alpha^{-\frac13}\eps \left(\tau^{-\frac13}\e^{-\nu\tau}+\eps\nu^{-\frac23}\right)(\norm{\widetilde{Q}_0(\tau)}_{H^{N}}+\norm{\widetilde{W}_0(\tau)}_{H^{N}}),
\end{equation}
having used that $\widetilde{U}_0^3=-\de_z^{-1}\de_y\widetilde{U}_0^2$ and the expression of $\widetilde{U}^2_0$ and $\widetilde{U}^1_0$ in terms of $Q$ and $W$, see \eqref{eq:U-QW}. For the nonzero modes, it follows that
\begin{align}
 \norm{(\overline{U^2_{\neq} U^3_{\neq}})_0(\tau)}_{H^{N}}+\norm{(\overline{U^2_{\neq} U^1_{\neq}})_0(\tau)}_{H^{N}}&\lesssim \norm{\sqrt{-\frac{\dot \cM}{\cM}}\cA Q_{\neq}(\tau)}_{H^N}\norm{\cA (Q,W)_{\neq}(\tau)}_{H^N}.
\end{align}
Substituting back in \eqref{eq:energy_doublezeros} and integrating in time leads to
\begin{align}    &\norm{\overline{U}^1_0(t)}_{H^{N}}^2+\norm{\overline{U}^3_0(t)}_{H^{N}}^2+\nu\int_0^t\norm{\de_y\overline{U}^1_0(\tau)}_{H^{N}}^2+\norm{\de_y\overline{U}^3_0(\tau)}_{H^{N}}^2\dd \tau  \notag\\
    &\qquad\lesssim\eps^2\cdot\alpha^{-\frac23}(\eps^2\nu^{-\frac43}+\eps^4\nu^{-\frac{10}{3}})+\eps^2\cdot \eps^2\nu^{-1}.
  \end{align} 
This concludes the proof.  
\end{proof}

\subsection{Simple zero modes}\label{sec:simple_zero}
Thanks to the symmetric structure of the equations \eqref{eq:simplezero_QK} for the simple zero modes, energy estimates give control of their $H^N$ norms. We summarize this as follows:
\begin{proposition}\label{prop:simple0_finalbound}
    Consider the simple zero modes, satisfying \eqref{eq:simplezero_QK}. Under the assumptions of Theorem \ref{thm:bootstrap} we have
    \begin{equation} \label{eq:s0_full_bound}
    \norm{\widetilde{F}_0}^2_{L^\infty_tH^N}+\nu\norm{\nabla\widetilde{F}_0}^2_{L^2_tH^N}\lesssim (\nu^{-\frac56}\eps +\alpha^{-\frac13}\nu^{-\frac83}\eps^3)\eps^2,\qquad F\in\{Q,W\}.
    \end{equation}
\end{proposition}
This threshold for simple zero modes is dictated by nonlinear interactions involving double zero modes. It establishes part \ref{it:s0-btstrap} of Theorem \ref{thm:bootstrap}. 


\begin{proof}[Proof of Proposition \ref{prop:simple0_finalbound}]
An $H^N$ energy estimate for $\widetilde{Q}_0$ and $\widetilde{W}_0$ gives
    \begin{equation}\label{eq:s0_energy}
    \begin{aligned}
        &\frac{1}{2}\left(\norm{\widetilde{Q}_0(t)}^2_{H^N}+\norm{\widetilde{W}_0(t)}^2_{H^N}\right)+\nu\left(\norm{\nabla \widetilde{Q}_0}^2_{L^2_tH^N}+\norm{\nabla \widetilde{W}_0}^2_{L^2_tH^N}\right) \\  
        &\qquad= \frac{1}{2}\left(\norm{\widetilde{Q}_0(0)}^2_{H^N}+\norm{\widetilde{W}_0(0)}^2_{H^N}\right)- \int_0^t\l \widetilde{Q}_0,\de_y\Delta\widetilde{\mathcal{P}}_0(U,U)\r_{H^N} \\
        &\qquad\quad-\int_0^t\l \widetilde{Q}_0,\Delta\widetilde{\mathcal{T}}_0(U,U^2)\r_{H^N} +\sqrt{\frac{\beta}{\beta-1}}\int_0^t\l \widetilde{W}_0, \de_z|\nabla|\widetilde{\mathcal{T}}_0(U,U^1) \r_{H^N}.
    \end{aligned}    
    \end{equation}
    By assumption \eqref{eq:initial_data} on the initial data, we have that
    \begin{equation}
        \frac{1}{2}\left(\norm{\widetilde{Q}_0(0)}^2_{H^N}+\norm{\widetilde{W}_0(0)}^2_{H^N}\right)\leq \eps^2,
    \end{equation}
    and thus to prove Proposition \ref{prop:simple0_finalbound} it remains to show suitable bounds on the last three terms in \eqref{eq:s0_energy}. These are provided in Lemmas \ref{lem:s0_pressure}, \ref{lem:s0_Qtransport} and \ref{lem:s0_Wtransport} below and yield that
    \begin{align}
        &\frac{1}{2}\left(\norm{\widetilde{Q}_0(t)}^2_{H^N}+\norm{\widetilde{W}_0(t)}^2_{H^N}\right)+\nu\left(\norm{\nabla \widetilde{Q}_0}^2_{L^2_tH^N}+\norm{\nabla \widetilde{W}_0}^2_{L^2_tH^N}\right)-\frac{1}{2}\left(\norm{\widetilde{Q}_0(0)}^2_{H^N}+\norm{\widetilde{W}_0(0)}^2_{H^N}\right)\notag\\
        &\qquad \lesssim (\nu^{-\frac56}\eps+\alpha^{-\frac13}\nu^{-\frac83}\eps^3)\eps^2.
    \end{align}    
\end{proof}

We begin by treating the terms involving the contributions from the pressure. 

\begin{lemma}\label{lem:s0_pressure}
Under the assumptions of Theorem \ref{thm:bootstrap} there holds that
    \begin{equation}\label{eq:s0_pressure_bound}
            \int_0^T\abs{\l \widetilde{Q}_0,\de_y\Delta\widetilde{\mathcal{P}}_0(U,U)\r_{H^N}} = \int_0^T\abs{\sum_{1\leq i,j\leq 3}\l \de_y \widetilde{Q}_0,(\de_i^LU^j\de_j^LU^i)_0\r_{H^N}}\lesssim (\nu^{-\frac56}\eps+\alpha^{-\frac13}\nu^{-\frac83}\eps^3)\eps^2.
        \end{equation}
\end{lemma}

\begin{proof}
To establish this lemma, it suffices to prove the bounds
\begin{align}
    \int_0^T \left|\l\de_y \widetilde{Q}_0, (\de_{i}U^{j}_{\neq}\de_{j}U^{i}_{\neq})_0\r_{H^N} \right|&\lesssim (\nu^{-\frac23}\eps)\eps^2, \qquad i,j\in\{1,3\},\label{eq:s0_pressure_bound_1}\\
    \int_0^T \left|\l \de_y\widetilde{Q}_0, (\de_y^LU^2_{\neq}\de_y^LU^2_{\neq})_0\r_{H^N} \right|& \lesssim (\nu^{-\frac23}\eps)\eps^2,\label{eq:s0_pressure_bound_2}\\
    \int_0^T \left|\l \de_y \widetilde{Q}_0, (\de_y^LU^{j}_{\neq}\de_{j}U^2_{\neq})_0\r_{H^N} \right|&\lesssim (\nu^{-\frac56}\eps)\eps^2,\qquad j\in\{1,3\},\label{eq:s0_pressure_bound_3}\\
    \int_0^T \left|\l \de_y \widetilde{Q}_0,(\de_{i}\widetilde{U}^{j}_{0}\de_{j}\widetilde{U}^{i}_{0})\r_{H^N} \right|&\lesssim (\alpha^{-\frac13}\nu^{-\frac23}\eps +\alpha^{-\frac13}\nu^{-\frac53}\eps^2)\eps^2,\qquad i,j\in\{2,3\},\label{eq:s0_pressure_bound_4}\\
    \int_0^T \left|\l\de_y \widetilde{Q}_0,(\de_y\overline{U}^{3}_{0}\de_z\widetilde{U}^2_{0})\r_{H^N} \right|&\lesssim (\nu^{-\frac32}\eps^2 +\alpha^{-\frac13}\nu^{-\frac83}\eps^3)\eps^2,\label{eq:s0_pressure_bound_5}
\end{align}
since the remaining terms in \eqref{eq:s0_pressure_bound} vanish as a consequence of $\de_xF_0=0$ and $\de_z\overline{F}_0=0$.

Using \eqref{eq:reconstruction_m} in Lemma \ref{lemma:additional_estim_U} and the bootstrap assumptions, for $i,j\in\{1,3\}$, we obtain \eqref{eq:s0_pressure_bound_1} as
\begin{align}
    \int_0^T \left|\l \de_y \widetilde{Q}_0,(\de_{i}U^{j}_{\neq}\de_{j}U^{i}_{\neq})_0\r_{H^N} \right|
    &\lesssim \int_0^T \norm{\de_y\widetilde{Q}_0}_{H^N} \norm{\de_{i}U^{j}_{\neq}\de_{j}U^{i}_{\neq}}_{H^N}  \\
    &\lesssim \int_0^T \norm{\nabla \widetilde{Q}_0}_{H^N} \norm{\de_{i}U^{j}_{\neq}}_{H^N}\norm{\de_{j}U^{i}_{\neq}}_{H^N}  \\
    &\lesssim \int_0^T \norm{\nabla \widetilde{Q}_0}_{H^N} \norm{\cA(Q,W)_{\neq}}_{H^N}^2  \\
    &\lesssim \norm{\nabla \widetilde{Q}_0}_{L^2_tH^N}\norm{\cA(Q,W)_{\neq}}_{L^2_tH^N}\norm{\cA(Q,W)_{\neq}}_{L^\infty_t H^N}\\
    &\lesssim(\nu^{-\frac12}\eps)(\nu^{-\frac16}\eps)\eps\\
    &\lesssim(\nu^{-\frac23}\eps)\eps^2.
\end{align}
Similarly, from \eqref{eq:absorption_nablaL} in Lemma \ref{lemma:additional_estim_U} the bound \eqref{eq:s0_pressure_bound_2} can be deduced for the interaction $i=j=2$, i.e.\
\begin{align}
    \int_0^T \left|\l \de_y \widetilde{Q}_0, (\de_y^LU^2_{\neq}\de_y^LU^2_{\neq})_0\r_{H^N} \right|&\lesssim \norm{\nabla \widetilde{Q}_0}_{L^2_tH^N}\norm{\cA Q_{\neq}}_{L^2_tH^N}\norm{\cA Q_{\neq}}_{L^\infty_t H^N}\lesssim (\nu^{-\frac23}\eps)\eps^2.
\end{align}
For the remaining nonzero modes interactions \eqref{eq:s0_pressure_bound_3} with $i=2$ and $j\in\{1,3\}$, we use \eqref{eq:reconstruction_m} and \eqref{eq:extratimedecay}, before concluding via the bootstrap assumptions that
\begin{align}
    \int_0^T \left|\l \de_y \widetilde{Q}_0,(\de_y^LU^{j}_{\neq}\de_{j}U^2_{\neq})_0\r_{H^N} \right|
    &\lesssim \norm{\nabla \widetilde{Q}_0}_{L^2_tH^N}\nu^{-\frac13}\norm{\cA(Q,W)_{\neq}}_{L^\infty_t H^N}\norm{\sqrt{-\frac{\dot \cM}{\cM}}\cA Q_{\neq}}_{L^2_tH^N}\\
    &\lesssim (\nu^{-\frac56}\eps)\eps^2.
\end{align}
For the next terms in \eqref{eq:s0_pressure_bound_4} -- involving simple and double zero interactions -- enhanced dissipation is no longer available and instead we rely on dispersion: Thanks to Proposition \ref{prop:nonlinear_dispersive_estimates} and \eqref{eq:simple0_conversion}, we have for $i,j\in\{2,3\}$ that
\begin{align}
    \int_0^T\left| \l \de_y \widetilde{Q}_0,(\de_{i}\widetilde{U}^{j}_{0}\de_{j}\widetilde{U}^{i}_{0})\r_{H^N} \right|&\lesssim \int_0^T \norm{\nabla \widetilde{Q}_0}_{H^N}\norm{\widetilde{Q}_0}_{H^N}\norm{\widetilde{Q}_0}_{L^\infty}  \\
    &\lesssim \int_0^T \norm{\nabla \widetilde{Q}_0}_{H^N}\norm{\widetilde{Q}_0}_{H^N}(\alpha^{-\frac13}t^{-\frac13}\e^{-\nu t}\eps +\alpha^{-\frac13}\nu^{-\frac23}\eps^2)\dd t  \\
    &\lesssim \norm{\nabla \widetilde{Q}_0}_{L^2_tH^N}\norm{\widetilde{Q}_0}_{L^\infty_t H^N}(\alpha^{-\frac13}\nu^{-\frac16}\eps) +\norm{\nabla \widetilde{Q}_0}_{L^2_tH^N}^2(\alpha^{-\frac13}\nu^{-\frac23}\eps^2) \\
    &\lesssim (\nu^{-\frac12}\eps)\eps (\alpha^{-\frac13}\nu^{-\frac16}\eps) + (\nu^{-1}\eps^2) (\alpha^{-\frac13}\nu^{-\frac23}\eps^2)\\
    &\lesssim(\alpha^{-\frac13}\nu^{-\frac23}\eps +\alpha^{-\frac13}\nu^{-\frac53}\eps^2)\eps^2,
\end{align}
and for the final term in \eqref{eq:s0_pressure_bound_5} we use Proposition \ref{prop:doublezero_estimates} and \eqref{eq:simple0_conversion} to get
\begin{align}
\int_0^T\left|  \l \de_y\widetilde{Q}_0,(\de_y\overline{U}^{3}_{0}\de_z\widetilde{U}^2_{0})\r_{H^N} \right|&\lesssim \int_0^T \norm{\nabla \widetilde{Q}_0}_{H^N} \norm{\de_y \overline{U}^3_0}_{H^N}\norm{\widetilde{Q}_0}_{H^N} ,\\
&\lesssim \norm{\nabla \widetilde{Q}_0}_{L^2_tH^N} \norm{\de_y \overline{U}^3_0}_{L^2_tH^N}\norm{\widetilde{Q}_0}_{L^\infty_t H^N}\\
&\lesssim (\nu^{-\frac12}\eps) \left(\nu^{-\frac12}(\nu^{-\frac12}\eps+\alpha^{-\frac13}\nu^{-\frac53}\eps^2)\eps\right)\eps\\
&\lesssim (\nu^{-\frac32}\eps^2 +\alpha^{-\frac13}\nu^{-\frac83}\eps^3)\eps^2.
\end{align}
This concludes the proof.
\end{proof}

We proceed with the first of the transport nonlinearities appearing in \eqref{eq:s0_energy}.

\begin{lemma}\label{lem:s0_Qtransport}
Under the assumptions of Theorem \ref{thm:bootstrap} we have the bound 
    \begin{equation}\label{eq:s0_Qtransport_bd}
      \int_0^T\abs{\l \widetilde{Q}_0,\Delta\widetilde{\mathcal{T}}_0(U,U^2)\r_{H^N} } = \int_0^T\abs{\l \nabla \widetilde{Q}_0,\nabla(U\cdot\nabla_L U^2)_0\r_{H^N} }\lesssim (\nu^{-\frac56}\eps+\alpha^{-\frac13}\nu^{-\frac83}\eps^3)\eps^2.
    \end{equation}
\end{lemma}
\begin{proof}
To establish \eqref{eq:s0_Qtransport_bd}, it suffices to prove the bounds
    \begin{align}
            \int_0^T\left| \l \nabla \widetilde{Q}_0,\nabla(U_{\neq}\cdot \nabla_LU^2_{\neq})_0\r_{H^N} \right|&\lesssim(\nu^{-\frac56}\eps)\eps^2+(\nu^{-\frac23}\eps)\eps^2,\label{eq:Ts0nn}\\
            \int_0^T\left| \l \nabla\widetilde{Q}_0,\nabla(\widetilde{U}^{j}_{0}\de_{j} \widetilde{U}^2_{0})\r_{H^N} \right|&\lesssim (\alpha^{-\frac13}\nu^{-\frac23}\eps +\alpha^{-\frac13}\nu^{-\frac53}\eps^2)\eps^2,\qquad j\in\{2,3\},\label{eq:Ts0s0s0}\\
            \int_0^T\left| \l\nabla \widetilde{Q}_0,\nabla(\overline{U}^{j}_{0}\de_{j} \widetilde{U}^2_{0})\r_{H^N} \right|&\lesssim (\nu^{-\frac32}\eps^2 +\alpha^{-\frac13}\nu^{-\frac83}\eps^3)\eps^2, \qquad j\in\{1,3\}.\label{eq:Ts000s0}
        \end{align}     
 For \eqref{eq:Ts0nn}, we use that
 \begin{equation}
     \nabla(U^j_{\neq}\de_j^L U^2_{\neq})_0=\nabla_L(U^j_{\neq}\de_j^L U^2_{\neq})_0=(\nabla_LU^j_{\neq}\de_j^L U^2_{\neq})_0+(U^j_{\neq}\nabla_L\de_j^L U^2_{\neq})_0,\qquad 1\leq j\leq 3.
 \end{equation}
 For $j\in\{1,3\}$ we have by \eqref{eq:absorption_nablaL} and \eqref{eq:extratimedecay} that
 \begin{equation}
  \norm{(\nabla_LU^j_{\neq}\de_j U^2_{\neq})_0}_{H^N}\lesssim \norm{\nabla_LU^j_{\neq}}_{H^N}\norm{\de_j U^2_{\neq}}_{H^N}\lesssim \nu^{-\frac13}\norm{\cA(Q,W)_{\neq}}_{H^N}\norm{\sqrt{-\frac{\dot \cM}{\cM}}\cA Q_{\neq}}_{H^N}
 \end{equation}
 and
 \begin{equation}
    \norm{(U^j_{\neq}\nabla_L\de_j U^2_{\neq})_0}_{H^N}\lesssim \norm{U^j_{\neq}}_{H^N}\norm{\nabla_L\de_j U^2_{\neq}}_{H^N}\lesssim\norm{\cA(Q,W)_{\neq}}_{H^N}\norm{\cA Q_{\neq}}_{H^N},
 \end{equation}
 and similarly for $j=2$ there holds that
 \begin{equation}
 \begin{aligned}
   \norm{\nabla_L(U^2_{\neq}\de_y^L U^2_{\neq})_0}_{H^N}&\lesssim  \norm{\nabla_LU^2_{\neq}}_{H^N}^2+\norm{U^2_{\neq}}_{H^N}\norm{\abs{\nabla_L}^2U^2_{\neq}}_{H^N}\\
   &\lesssim \norm{\cA 
 Q_{\neq}}_{H^N}^2 + \nu^{-\frac13}\norm{\cA 
 Q_{\neq}}_{H^N}\norm{\sqrt{-\frac{\dot \cM}{\cM}}\cA Q_{\neq}}_{H^N},
 \end{aligned}  
 \end{equation}
 so that altogether
 \begin{equation}
 \begin{aligned}
 \int_0^T\left| \l \nabla \widetilde{Q}_0,\nabla(U_{\neq}\cdot\nabla_L U^2_{\neq})_0\r_{H^N} \right|&\lesssim \norm{\nabla \widetilde{Q}_0}_{ L^2_tH^N}\nu^{-\frac13}\norm{\cA(Q,W)_{\neq}}_{L^\infty_t H^N}\norm{\sqrt{-\frac{\dot \cM}{\cM}}\cA Q_{\neq}}_{ L^2_tH^N}\\
    &\quad + \norm{\nabla \widetilde{Q}_0}_{ L^2_tH^N}\norm{\cA(Q,W)_{\neq}}_{L^\infty_t H^N}\norm{\cA Q_{\neq}}_{ L^2_tH^N}\\
    &\lesssim (\nu^{-\frac56}\eps)\eps^2+(\nu^{-\frac23}\eps)\eps^2.
 \end{aligned}
 \end{equation}
 For the interactions in \eqref{eq:Ts0s0s0} between simple zero modes we will use the dispersive decay estimates from Proposition \ref{prop:nonlinear_dispersive_estimates}: these imply that
\begin{align}
 \sum_{j\in\{2,3\}}\norm{\nabla (\widetilde{U}^{j}_{0}\de_{j} \widetilde{U}^2_{0})}_{H^N} &\lesssim \sum_{j\in\{2,3\}}\norm{\nabla \widetilde{U}^{j}_{0}}_{L^\infty}\norm{\de_{j} \widetilde{U}^2_{0}}_{H^N} + \norm{\nabla \widetilde{U}^{j}_{0}}_{H^N}\norm{\de_{j} \widetilde{U}^2_{0}}_{L^\infty} \\
 &\quad\qquad\qquad +\norm{\widetilde{U}^{j}_{0}}_{L^\infty}\norm{\nabla \de_{j} \widetilde{U}^2_{0}}_{H^N}+\norm{\widetilde{U}^{j}_{0}}_{H^N}\norm{\nabla \de_{j} \widetilde{U}^2_{0}}_{L^\infty}\\
 &\lesssim (\alpha^{-\frac13}t^{-\frac13}\e^{-\nu t}\eps +\alpha^{-\frac13}\nu^{-\frac23}\eps^2)\norm{\widetilde{Q}_0}_{H^N},
\end{align}  
and thus 
\begin{align}
    \int_0^T\left| \l \nabla \widetilde{Q}_0,\nabla (\widetilde{U}^{i}_{0}\de_{i} \widetilde{U}^2_{0})\r_{H^N} \right|&\lesssim \int_0^T \norm{\nabla \widetilde{Q}_0}_{H^N}\norm{\widetilde{Q}_0}_{H^N}(\alpha^{-\frac13}t^{-\frac13}\e^{-\nu t}\eps +\alpha^{-\frac13}\nu^{-\frac23}\eps^2)\dd t  \\
    &\lesssim(\alpha^{-\frac13}\nu^{-\frac23}\eps +\alpha^{-\frac13}\nu^{-\frac53}\eps^2)\eps^2.
\end{align}

Finally, for the terms \eqref{eq:Ts000s0} involving double zero interactions, we apply the bounds from Proposition \ref{prop:doublezero_estimates}, together with \eqref{eq:simple0_conversion}, to deduce that when $j\in\{1,3\}$
\begin{align}
    \int_0^T\left| \l \nabla\widetilde{Q}_0,\nabla(\overline{U}^{j}_{0}\de_{j} \widetilde{U}^2_{0})\r_{H^N} \right|&\lesssim \int_0^T \norm{\nabla \widetilde{Q}_0}_{H^N} \left(\norm{ \overline{U}^{j}_0}_{H^N}+\norm{\de_y \overline{U}^{j}_0}_{H^N}\right)\norm{\widetilde{Q}_0}_{H^N} \\
    &\lesssim \norm{\nabla \widetilde{Q}_0}_{ L^2_tH^N} \left(\norm{ \overline{U}^{j}_0}_{L^\infty_t H^N}\norm{\nabla \widetilde{Q}_0}_{L^2 H^N}+\norm{\de_y \overline{U}^{j}_0}_{ L^2_tH^N}\norm{ \widetilde{Q}_0}_{L^\infty_t H^N}\right)\\
    &\lesssim (\nu^{-\frac32}\eps^2 +\alpha^{-\frac13}\nu^{-\frac83}\eps^3)\eps^2.
\end{align}
The proof is over.
\end{proof}

Lastly, we conclude the estimates with the second transport nonlinearity in \eqref{eq:s0_energy}.

\begin{lemma}\label{lem:s0_Wtransport}
Under the assumptions of Theorem \ref{thm:bootstrap} we have that
\begin{equation}\label{eq:s0_Wtransport_bd}
    \int_0^T\abs{\l \widetilde{W}_0, \de_z|\nabla|\widetilde{\mathcal{T}}_0(U,U^1) \r_{H^N} } =\int_0^T\abs{\l |\nabla|\widetilde{W}_0, \de_z(U\cdot\nabla_LU^1)_0 \r_{H^N} }\lesssim (\nu^{-\frac56}\eps+\alpha^{-\frac13}\nu^{-\frac83}\eps^3)\eps^2.
    \end{equation}
\end{lemma}

\begin{proof}
We will show the following bounds, which imply \eqref{eq:s0_Wtransport_bd}:
\begin{align}
    \int_0^T\left| \l |\nabla|\widetilde{W}_0,\de_z(U^{j}_{\neq}\de_{j} U^1_{\neq})_0\r_{H^N} \right|&\lesssim (\nu^{-\frac23}\eps)\eps^2, \qquad j\in\{1,3\},\label{eq:s0_Wtransport_bd1}\\
    \int_0^T\left| \l|\nabla| \widetilde{W}_0,\de_z(U^2_{\neq}\de_y^L U^1_{\neq})_0\r_{H^N} \right|&\lesssim (\nu^{-\frac56}\eps)\eps^2,\label{eq:s0_Wtransport_bd2}\\
    \int_0^T\left| \l |\nabla|\widetilde{W}_0,\de_z(\widetilde{U}^{j}_{0}\de_{j} \widetilde{U}^1_{0})\r_{H^N} \right|&\lesssim (\alpha^{-\frac13}\nu^{-\frac23}\eps +\alpha^{-\frac13}\nu^{-\frac53}\eps^2)\eps^2,\qquad j\in\{2,3\},\label{eq:s0_Wtransport_bd3}\\
    \int_0^T\left| \l |\nabla|\widetilde{W}_0,\de_z(\overline{U}^3_{0}\de_z \widetilde{U}^1_{0})\r_{H^N} \right|&\lesssim (\nu^{-\frac32}\eps^2 +\alpha^{-\frac13}\nu^{-\frac83}\eps^3)\eps^2,\label{eq:s0_Wtransport_bd4}\\
    \int_0^T\left| \l |\nabla|\widetilde{W}_0,\de_z(\widetilde{U}^2_{0}\de_y \overline{U}^1_{0})\r_{H^N} \right|&\lesssim (\nu^{-\frac32}\eps^2 +\alpha^{-\frac13}\nu^{-\frac83}\eps^3)\eps^2.\label{eq:s0_Wtransport_bd5}
\end{align}
These estimates are proved similarly as those in the preceding two lemmas, and hence our proof briefly highlights the main inequalities.

Appealing to \eqref{eq:reconstruction_m} in Lemma \ref{lemma:additional_estim_U}, for \eqref{eq:s0_Wtransport_bd1} we have for $j\in\{1,3\}$ that
\begin{align}
    \int_0^T\left| \l |\nabla|\widetilde{W}_0,\de_z(U^{j}_{\neq}\de_{j} U^1_{\neq})_0\r_{H^N} \right|&\lesssim \norm{\nabla \widetilde{W}_0}_{L^2_tH^N}\norm{\cA (Q,W)_{\neq}}_{L^2_tH^N}\norm{\cA(Q,W)_{\neq}}_{L^\infty_t H^N}\\
    &\lesssim (\nu^{-\frac23}\eps)\eps^2,
\end{align}
and with \eqref{def:multi-M2} and \eqref{eq:absorption_nablaL} we also get \eqref{eq:s0_Wtransport_bd2} via
\begin{align}
    \int_0^T\left| \l|\nabla| \widetilde{W}_0,\de_z(U^2_{\neq}\de_y^L U^1_{\neq})_0\r_{H^N} \right|&\lesssim \norm{\nabla \widetilde{W}_0}_{L^2_tH^N}\norm{\sqrt{-\frac{\dot \cM}{\cM}}\cA Q_{\neq}}_{L^2_tH^N}\nu^{-\frac13}\norm{\cA(Q,W)_{\neq}}_{L^\infty_t H^N}\\
    &\lesssim (\nu^{-\frac56}\eps)\eps^2. 
\end{align}
To bound the interaction between simple zero modes in \eqref{eq:s0_Wtransport_bd3} we invoke Proposition \ref{prop:nonlinear_dispersive_estimates}: for $j\in\{2,3\}$ we obtain
\begin{align}
    \int_0^T\left| \l |\nabla|\widetilde{W}_0,\de_z(\widetilde{U}^{j}_{0}\de_{j} \widetilde{U}^1_{0})\r_{H^N} \right|&\lesssim \int _0^\infty \norm{\nabla \widetilde{W}_0}_{H^N}\norm{(\widetilde{Q}_0,\widetilde{W}_0)}_{H^N}(\alpha^{-\frac13}t^{-\frac13}\e^{-\nu t}\eps +\alpha^{-\frac13}\nu^{-\frac23}\eps^2) \dd t \\
    &\lesssim(\alpha^{-\frac13}\nu^{-\frac23}\eps +\alpha^{-\frac13}\nu^{-\frac53}\eps^2)\eps^2.
\end{align}
Finally, for the interactions with double zero modes in \eqref{eq:s0_Wtransport_bd4}, \eqref{eq:s0_Wtransport_bd5} we use Proposition \ref{prop:doublezero_estimates}. Consequently,
\begin{align}
    \int_0^T\left| \l |\nabla|\widetilde{W}_0,\de_z(\overline{U}^3_{0}\de_z \widetilde{U}^1_{0})\r_{H^N} \right|&\lesssim \int_0^T \norm{\nabla \widetilde{W}_0}_{H^N}\norm{ \overline{U}^{3}_0}_{H^N}\norm{\nabla \widetilde{W}_0}_{H^N} \\
    &\lesssim \norm{\nabla \widetilde{W}_0}^2_{L^2_tH^N}\norm{ \overline{U}^{3}_0}_{L^\infty_t H^N}\\
    &\lesssim (\nu^{-\frac32}\eps^2 +\alpha^{-\frac13}\nu^{-\frac83}\eps^3)\eps^2,
\end{align}
and lastly
\begin{align}
    \int_0^T\left| \l |\nabla|\widetilde{W}_0,\de_z(\widetilde{U}^2_{0}\de_y \overline{U}^1_{0})\r_{H^N} \right|&\lesssim \int_0^T \norm{\nabla \widetilde{W}_0}_{H^N} \norm{\de_y \overline{U}^{1}_0}_{H^N}\norm{\widetilde{Q}_0}_{H^N} \\
    &\lesssim \norm{\nabla \widetilde{W}_0}_{L^2_tH^N}\norm{\de_y \overline{U}^{1}_0}_{L^2_tH^N}\norm{\widetilde{Q}_0}_{L^\infty_t H^N}\\
    &\lesssim (\nu^{-\frac32}\eps^2 +\alpha^{-\frac13}\nu^{-\frac83}\eps^3)\eps^2.
\end{align}
The lemma is proved.
\end{proof}

\subsection{Nonzero modes}\label{sec:non-zero}
For the nonzero modes $Q_{\neq}, W_{\neq}$ we recall from \eqref{eq:nonlin} that they satisfy the equations
\begin{equation}\label{eq:nonlin_nonzero}
\begin{cases}
        \de_t Q_{\neq} +\alpha \de_z|\nabla_L|^{-1}W_{\neq} -\nu \Delta_L Q_{\neq} = -\de_y^L\Delta_L\mathcal{P}_{\neq}(U,U)-\Delta_L\mathcal{T}_{\neq}(U,U^2),\\
        \de_t W_{\neq}-\de_{x}\de_y^L\abs{\nabla_L}^{-2}W_{\neq} +\alpha \de_z|\nabla_L|^{-1}Q_{\neq}-\nu\Delta_L W_{\neq}=\sqrt{\frac{\beta}{\beta-1}}\abs{\nabla_L}\left(\de_z\mathcal{T}_{\neq}(U,U^1)-\de_x\cT_{\neq}(U,U^3)\right),
    \end{cases}
\end{equation}
Part \ref{it:non0-btstrap} of Theorem \ref{thm:bootstrap} then follows directly from the following proposition.  
\begin{proposition}\label{prop:estim_nonzero}
    There exists $C>0$ such that under the assumptions of Theorem \ref{thm:bootstrap} we have that
    \begin{equation}
        \sum_{F\in\{Q,W\}} \norm{\cA F_{\neq}}^2_{L^\infty_t H^N}+\nu\norm{\nabla_L\cA F_{\neq}}^2_{L^2_tH^N}+\norm{\sqrt{-\frac{\dot \cM}{ \cM}}\cA F_{\neq}}^2_{L^2_tH^N}\leq \eps^2+ C(\nu^{-\frac56}\eps+\alpha^{-\frac13}\nu^{-\frac83}\eps^3)\eps^2.
    \end{equation}
\end{proposition}

\begin{proof}
We start with an energy estimate for the norm $\sum_{F\in\{Q,W\}} \norm{\cA F_{\neq}}^2_{H^N}$ of solutions to the system \eqref{eq:nonlin_nonzero} and obtain
\begin{equation}\label{eq:nonzero_energy}
    \begin{aligned}
        &\frac12 \left( \norm{\cA Q_{\neq}(t)}^2_{H^N}+\norm{\cA W_{\neq}(t)}^2_{H^N}\right) +\norm{\sqrt{-\frac{\dot \cM}{\cM}}\cA Q_{\neq}}^2_{L^2_tH^N}+\norm{\sqrt{-\frac{\dot \cM}{\cM}}\cA W_{\neq}}^2_{L^2_tH^N} \\
        &\quad + \norm{\sqrt{-\frac{\dot \fm}{\fm}}\cA Q_{\neq}}^2_{L^2_tH^N}+\norm{\sqrt{-\frac{\dot \fm}{\fm}}\cA W_{\neq}}^2_{L^2_tH^N}+ \nu \norm{\nabla_L\cA Q_{\neq}}^2_{L^2_tH^N}+\nu\norm{\nabla_L\cA W_{\neq}}^2_{L^2_tH^N}\\
        &= \frac12 \left( \norm{\cA Q_{\neq}(0)}^2_{H^N}+\norm{\cA W_{\neq}(0)}^2_{H^N}\right) \\
        &\quad + \delta\nu^\frac13\left( \norm{\cA Q_{\neq}(t)}^2_{L^2_tH^N}+\norm{\cA W_{\neq}(t)}^2_{L^2_tH^N}\right)+ \int_0^t \l \cA W_{\neq}, \cA\de_{x}\de_y^L\abs{\nabla_L}^{-2}W_{\neq}\r_{H^N} \\
        &\quad +\int_0^t N_{E}  ,
    \end{aligned}
\end{equation}
where
\begin{equation}\label{eq:def_NE}
\begin{aligned}
        N_{E}&=- \l \cA Q_{\neq},\cA \de_y^L\Delta_L\cP_{\neq}(U,U)\r_{H^N} -\l \cA Q_{\neq},\cA\Delta_L\cT_{\neq}(U,U^2)\r_{H^N}\\
        &\qquad +\sqrt{\frac{\beta}{\beta-1}}\left(\l \cA W_{\neq},\cA \abs{\nabla_L}\de_z\mathcal{T}_{\neq}(U,U^1)\r_{H^N}-\l \cA W_{\neq},\cA \abs{\nabla_L}\de_x\cT_{\neq}(U,U^3)\r_{H^N}\right).
\end{aligned}
\end{equation}
By assumption \eqref{eq:initial_data} we have that 
\begin{equation}
  \frac12 \left( \norm{\cA Q_{\neq}(0)}^2_{H^N}+\norm{\cA W_{\neq}(0)}^2_{H^N}\right)\leq \eps^2,  
\end{equation}
and by Lemma \ref{lem:more_enh_dissip_neq} we can choose $\delta>0$ so small that
\begin{equation}\label{eq:delta-choice}
    \delta\nu^\frac13\left( \norm{\cA Q_{\neq}(t)}^2_{L^2_tH^N}+\norm{\cA W_{\neq}(t)}^2_{L^2_tH^N}\right)\leq \eps^2.
\end{equation}
Moreover, arguing as in the linear analysis in \eqref{eq:nonzero_linear_energy} in Section \ref{ssec:nonzero_linear}, we have that
\begin{equation}
    \abs{\l \cA W_{\neq}, \cA\de_{x}\de_y^L\abs{\nabla_L}^{-2}W_{\neq}\r_{H^N}}\leq\frac{\nu}{2}\norm{\nabla_L\cA W_{\neq}}_{H^N}^2+\norm{\sqrt{-\frac{\dot \fm}{\fm}}\cA W_{\neq}}_{H^N}^2.
\end{equation}
To prove the claim it thus suffices to show that
\begin{equation}
 \int_0^T \abs{N_E}\lesssim (\nu^{-\frac56}\eps+\alpha^{-\frac13}\nu^{-\frac83}\eps^3)\eps^2.
\end{equation}
This is done separately for the four terms in $N_E$, see \eqref{eq:def_NE}, in Lemmas \ref{lem:neq_pressure}--\ref{lem:neq_Wtransport} below.
\end{proof}

The proof of this and the following lemmas expands on arguments as already encountered in Section \ref{sec:simple_zero}, with heavy reliance on the structure and bounds for our multipliers. In particular, we will make use of the product estimate \eqref{eq:A-product-bd} in Lemma \ref{lem:m-product} for $\fm$ resp.\ $\cA$. Moreover, we will use \eqref{eq:A-Sobbd} repeatedly and crucially the bounds for $U$ in terms of $(Q,W)$ established in Lemma \ref{lemma:additional_estim_U}. 
\begin{lemma}\label{lem:neq_pressure}
Under the assumptions of Theorem \ref{thm:bootstrap} we have the bound
\begin{equation}
\begin{aligned}
   \int_0^T\abs{\l \cA Q_{\neq},\cA \de_y^L\Delta_L\cP_{\neq}(U,U)\r_{H^N}} & = \int_0^T \abs{\sum_{1\leq i,j\leq 3} \l \cA\de_y^LQ_{\neq},\cA (\de_i^LU^j\de_j^LU^i)_{\neq}\r_{H^N}}\\
   &\lesssim (\nu^{-\frac56}\eps +\alpha^{-\frac13}\nu^{-\frac83}\eps^3)\eps^2.
\end{aligned}   
\end{equation}
\end{lemma}

\begin{proof}
We will prove the following bounds, from which Lemma \ref{lem:neq_pressure} follows:
\begin{align}
    \int_0^T\left|  \l \cA\de_y^LQ_{\neq},\cA(\de_{i}U^{j}_{\neq}\de_{j}U^{i}_{\neq})\r_{H^N} \right|&\lesssim (\nu^{-\frac23}\eps)\eps^2 ,\qquad i,j\in\{1,3\},\label{eq:neq_pressure_bd1}\\
    \int_0^T\left|  \l \cA\de_y^LQ_{\neq},\cA (\de_y^LU^{j}_{\neq}\de_{j}U^2_{\neq})\r_{H^N} \right|&\lesssim (\nu^{-\frac56}\eps)\eps^2, \qquad j\in\{1,3\},\label{eq:neq_pressure_bd2}\\
    \int_0^T\left|  \l \cA\de_y^LQ_{\neq},\cA (\de_y^LU^2_{\neq}\de_y^LU^2_{\neq})\r_{H^N} \right|&\lesssim (\nu^{-\frac23}\eps)\eps^2 ,\label{eq:neq_pressure_bd3}\\
    \int_0^T\left|  \l \cA\de_y^LQ_{\neq},\cA (\de_z\widetilde{U}^{j}_{0}\de_{j}U^{3}_{\neq})\r_{H^N} \right|&\lesssim (\nu^{-\frac23}\eps)\eps^2, \qquad j\in\{1,3\},\label{eq:neq_pressure_bd4}\\
    \int_0^T\left|  \l \cA\de_y^LQ_{\neq},\cA (\de_y\widetilde{U}^{j}_{0}\de_{j}U^2_{\neq})\r_{H^N} \right|&\lesssim (\nu^{-\frac12}\eps)\eps^2, \qquad j\in\{1,3\},\label{eq:neq_pressure_bd5}\\
    \int_0^T\left|  \l \cA\de_y^LQ_{\neq},\cA (\de_y^LU^{3}_{\neq}\de_z\widetilde{U}^2_{0})\r_{H^N} \right|&\lesssim (\nu^{-\frac23}\eps)\eps^2,\label{eq:neq_pressure_bd6}\\
    \int_0^T\left|  \l \cA\de_y^LQ_{\neq},\cA (\de_y\widetilde{U}^2_{0}\de_y^LU^2_{\neq})\r_{H^N} \right|&\lesssim (\nu^{-\frac23}\eps)\eps^2,\label{eq:neq_pressure_bd7}\\
    \int_0^T\left|  \l \cA\de_y^LQ_{\neq},\cA (\de_y\overline{U}^{j}_{0}\de_{j}U^2_{\neq})\r_{H^N} \right|&\lesssim (\nu^{-\frac32}\eps^2 +\alpha^{-\frac13}\nu^{-\frac83}\eps^3)\eps^2, \qquad j\in\{1,3\}.\label{eq:neq_pressure_bd8}
\end{align}
For \eqref{eq:neq_pressure_bd1} we use \eqref{eq:lowerbound-m}, \eqref{eq:A-Sobbd} and \eqref{eq:reconstruction_m} to obtain that for $i,j\in\{1,3\}$
\begin{align}
    \int_0^T \left|\l \cA\de_y^LQ_{\neq},\cA(\de_{i}U^{j}_{\neq}\de_{j}U^{i}_{\neq})\r_{H^N}\right|&\leq \int_0^T \norm{\cA\de_y^LQ_{\neq}}_{H^N}\e^{\delta \nu^{\frac13}t}\norm{\de_{i}U^{j}_{\neq}}_{H^N}\norm{\de_{j}U^{i}_{\neq}}_{H^N}\dd t \\
    &\lesssim \norm{\nabla_L\cA  Q_{\neq}}_{L^2_tH^N}\norm{\cA(Q,W)_{\neq}}_{L^2_tH^N}\norm{\cA(Q,W)_{\neq}}_{L^\infty_t H^N}\\
    &\lesssim(\nu^{-2/3}\eps)\eps^2.
\end{align}
For \eqref{eq:neq_pressure_bd2}, we apply \eqref{eq:absorption_nablaL} and  \eqref{eq:extratimedecay}  
to deduce that when $j\in\{1,3\}$,
\begin{align} 
\int_0^T \left|\l \cA\de_y^LQ_{\neq},\cA(\de_y^LU^{j}_{\neq}\de_{j}U^2_{\neq})\r_{H^N}\right|&\lesssim \norm{\nabla_L\cA  Q_{\neq}}_{L^2_tH^N}\nu^{-\frac13}\norm{\cA(Q,W)_{\neq}}_{L^\infty_t H^N}\norm{\sqrt{-\frac{\dot \cM}{\cM}}\cA Q_{\neq}}_{L^2_t H^N}\notag\\
&\lesssim (\nu^{-\frac56}\eps)\eps^2. 
\end{align}
Similarly we obtain \eqref{eq:neq_pressure_bd3} by using \eqref{eq:absorption_nablaL} via
\begin{equation}
\int_0^T \left|\l \cA\de_y^LQ_{\neq},\cA(\de_y^LU^2_{\neq}\de_y^LU^2_{\neq})\r_{H^N}\right|\lesssim \norm{\nabla_L\cA  Q_{\neq}}_{L^2_tH^N}\norm{\cA Q_{\neq}}_{L^2_tH^N}\norm{\cA Q_{\neq}}_{L^\infty_t H^N}\lesssim (\nu^{-\frac23}\eps)\eps^2.
\end{equation}
For \eqref{eq:neq_pressure_bd4} we invoke \eqref{eq:reconstruction_m}  and \eqref{eq:simple0_conversion}, obtaining for $j\in\{1,3\}$ that
\begin{align}     
\int_0^T \left| \l \cA\de_y^LQ_{\neq},\cA(\de_z\widetilde{U}^{j}_{0}\de_{j}U^{3}_{\neq})\r_{H^N}\right|&\lesssim \norm{\nabla_L\cA  Q_{\neq}}_{L^2_tH^N}\norm{(\widetilde{Q}_0,\widetilde{W}_0)}_{L^\infty_t H^N}\norm{\cA(Q,W)_{\neq}}_{L^2_t H^N}\\
&\lesssim (\nu^{-\frac23}\eps)\eps^2,
\end{align}
and similarly, using also \eqref{eq:extratimedecay}, 
\begin{align} 
\int_0^T \left| \l \cA\de_y^LQ_{\neq},\cA (\de_y\widetilde{U}^{j}_{0}\de_{j}U^2_{\neq})\r_{H^N}\right|&\lesssim  \norm{\nabla_L\cA  Q_{\neq}}_{L^2_tH^N}\norm{(\widetilde{Q}_0,\widetilde{W}_0)}_{L^\infty_t H^N}\norm{\sqrt{-\frac{\dot \cM}{\cM}}\cA Q_{\neq}}_{L^2_tH^N}\\
&\lesssim (\nu^{-\frac12}\eps)\eps^2,
\end{align}
which establishes \eqref{eq:neq_pressure_bd5}. The bound \eqref{eq:neq_pressure_bd7} is obtained similarly as
\begin{equation} 
\int_0^T \left| \l \cA\de_y^LQ_{\neq},\cA (\de_y\widetilde{U}^2_{0}\de_y^LU^2_{\neq})\r_{H^N}\right|\lesssim \norm{\nabla_L\cA  Q_{\neq}}_{L^2_tH^N}\norm{\widetilde{Q}_0}_{L^\infty_t H^N}\norm{\cA Q_{\neq}}_{L^2_tH^N}\lesssim (\nu^{-\frac23}\eps)\eps^2.
\end{equation}
For the remaining term \eqref{eq:neq_pressure_bd6} involving a product of non-zero and simple zero modes, we use the product estimate \eqref{eq:A-product-bd} from Lemma \ref{lem:m-product} to obtain a low threshold: Together with \eqref{eq:U-QW} and \eqref{eq:simple0_conversion}, this gives that
\begin{align}      
\int_0^T \left| \l \cA\de_y^LQ_{\neq},\cA(\de_y^LU^{3}_{\neq}\de_z\widetilde{U}^2_{0})\r_{H^N}\right|&\lesssim \int_0^T \norm{ \cA\de_y^LQ_{\neq}}_{H^N}\norm{\cA\de_y^LU^{3}_{\neq}}_{H^N}\norm{\l\nabla\r\de_z\widetilde{U}^2_{0}}_{H^N}\notag\\
&\lesssim \norm{\nabla_L\cA  Q_{\neq}}_{L^2_tH^N}\norm{\cA(Q,W)_{\neq}}_{L^2_tH^N}\norm{\widetilde{Q}_0}_{L^\infty_t H^N}\notag\\
&\lesssim (\nu^{-\frac23}\eps)\eps^2.
\end{align} 
Finally we consider the double zero modes interactions in \eqref{eq:neq_pressure_bd8}: using \eqref{eq:reconstruction_m} and Proposition \ref{prop:doublezero_estimates}, for $j\in\{1,3\}$ we have the bound
\begin{align}\label{eq:doublezerosetim_simplezero}
\int_0^T \left| \l \cA\de_y^LQ_{\neq},\cA (\de_y\overline{U}^{j}_{0}\de_{j}U^2_{\neq})\r_{H^N}\right|&\lesssim \norm{\nabla_L\cA  Q_{\neq}}_{L^2_tH^N}\norm{\de_y\overline{U}^{j}_0}_{L^2_tH^N}\norm{\cA Q_{\neq}}_{L^\infty_t H^N}\notag\\
&\lesssim (\nu^{-\frac32}\eps^2 +\alpha^{-\frac13}\nu^{-\frac83}\eps^3)\eps^2.
\end{align}
The proof of the lemma is concluded.
\end{proof}

The next result concerns the first transport nonlinearity appearing in \eqref{eq:def_NE}.

\begin{lemma}\label{lem:neq_Qtransport}
Under the assumptions of Theorem \ref{thm:bootstrap} we have the bound
\begin{equation}
   \int_0^T\abs{\l \cA Q_{\neq},\cA\Delta_L\cT_{\neq}(U,U^2)\r_{H^N}} = \int_0^T \abs{\l \cA\nabla_LQ_{\neq},\cA\nabla_L(U\cdot \nabla_L U^2)_{\neq}\r_{H^N}}\lesssim (\nu^{-\frac56}\eps +\alpha^{-\frac13}\nu^{-\frac83}\eps^3)\eps^2.
\end{equation}
\end{lemma}

\begin{proof}
The lemma follows from the following more detailed bounds, which we will establish subsequently:
\begin{align}
    \int_0^T\left|  \l \cA\nabla_LQ_{\neq},\cA\nabla_L(U_{\neq}\cdot\nabla_L U^2_{\neq})\r_{H^N} \right|&\lesssim (\nu^{-\frac23}\eps)\eps^2+(\nu^{-\frac56}\eps)\eps^2,\label{eq:neq_Qtransport_bd1}\\ 
    \int_0^T\left|  \l \cA\nabla_LQ_{\neq},\cA\nabla_L(\widetilde{U}_{0}\cdot\nabla_L U^2_{\neq})\r_{H^N} \right|&\lesssim (\nu^{-\frac23}\eps)\eps^2,\label{eq:neq_Qtransport_bd2}\\ 
    \int_0^T\left|  \l \cA\nabla_LQ_{\neq},\cA\nabla_L(U^2_{\neq}\de_y \widetilde{U}^2_{0})\r_{H^N} \right|&\lesssim (\nu^{-\frac23}\eps)\eps^2,\label{eq:neq_Qtransport_bd3}\\ 
    \int_0^T\left|  \l \cA\nabla_LQ_{\neq},\cA\nabla_L(U^3_{\neq}\de_z \widetilde{U}^2_{0})\r_{H^N} \right|&\lesssim (\nu^{-\frac23}\eps)\eps^2,\label{eq:neq_Qtransport_bd4}\\
    \int_0^T\left|  \l \cA\nabla_LQ_{\neq},\cA\nabla_L(\overline{U}^{j}_{0}\de_{j} U^2_{\neq})\r_{H^N} \right|&\lesssim (\nu^{-\frac23}\eps + \nu^{-\frac32}\eps^2 +\alpha^{-\frac13}\nu^{-\frac83}\eps^3)\eps^2,\quad j\in\{1,3\}.\label{eq:neq_Qtransport_bd5}
\end{align}
For \eqref{eq:neq_Qtransport_bd1}, we have with \eqref{eq:A-Sobbd}, \eqref{eq:absorption_nablaL}, \eqref{eq:reconstruction_m}, and \eqref{eq:extratimedecay} that for $j\in\{1,3\}$
\begin{align}
    &\int_0^T \left|  \l \cA\nabla_LQ_{\neq},\cA\nabla_L(U^{j}_{\neq}\de_{j} U^2_{\neq})\r_{H^N}\right|\leq \int_0^T \norm{\nabla_L\cA Q_{\neq}}_{H^N} \e^{\delta \nu^{\frac13}t}\norm{\nabla_L(U^{j}_{\neq}\de_{j} U^2_{\neq})}_{H^N}\\
    &\leq \int_0^T \norm{\nabla_L\cA Q_{\neq}}_{H^N}\e^{\delta \nu^{\frac13}t}\left(\norm{\nabla_LU^{j}_{\neq}}_{H^N}\norm{\de_{j} U^2_{\neq}}_{H^N}+\norm{U^{j}_{\neq}}_{H^N}\norm{\nabla_L\de_{j} U^2_{\neq}}_{H^N}\right)\\
    &\lesssim \norm{\nabla_L\cA  Q_{\neq}}_{L^2_tH^N}\nu^{-\frac13}\norm{\cA(Q,W)_{\neq}}_{L^\infty_t H^N}\norm{\sqrt{-\frac{\dot \cM}{\cM}}\cA Q_{\neq}}_{L^2_tH^N}\\
    &\qquad +\norm{\nabla_L\cA  Q_{\neq}}_{L^2_tH^N}\norm{\cA(Q,W)_{\neq}}_{L^\infty_t H^N}\norm{\cA Q_{\neq}}_{L^2_tH^N}\\
    &\lesssim (\nu^{-\frac56}\eps)\eps^2+(\nu^{-\frac23}\eps)\eps^2,
\end{align}
as well as
\begin{align}
    &\int_0^T \left|  \l \cA\nabla_LQ_{\neq},\cA\nabla_L(U^2_{\neq}\de_y^L U^2_{\neq})\r_{H^N}\right|\\
    &\quad\leq \int_0^T \norm{\nabla_L\cA Q_{\neq}}_{H^N}\e^{\delta \nu^{\frac13}t}\left(\norm{\nabla_LU^{2}_{\neq}}_{H^N}\norm{\de_y^L U^2_{\neq}}_{H^N}+\norm{U^{2}_{\neq}}_{H^N}\norm{\nabla_L\de_y^L U^2_{\neq}}_{H^N}\right)\\
    &\quad\lesssim \norm{\nabla_L\cA  Q_{\neq}}_{L^2_tH^N}\norm{\cA Q_{\neq}}_{L^\infty_t H^N}\norm{\cA Q_{\neq}}_{L^2_tH^N}\\
    &\qquad +\norm{\nabla_L\cA  Q_{\neq}}_{L^2_tH^N}\norm{\sqrt{-\frac{\dot \cM}{\cM}}\cA Q_{\neq}}_{L^2_t H^N}\nu^{-\frac13}\norm{\cA Q_{\neq}}_{L^\infty_t H^N}\\
    &\quad\lesssim (\nu^{-\frac23}\eps)\eps^2+(\nu^{-\frac56}\eps)\eps^2.
\end{align}
For \eqref{eq:neq_Qtransport_bd2} we use \eqref{eq:simple0_conversion} and \eqref{eq:reconstruction_m}, and in case two derivatives fall on $U^2_{\neq}$ also the product estimate \eqref{eq:A-product-bd} for $\cA$ to obtain that
\begin{align}
    \int_0^T \left|  \l \cA\nabla_LQ_{\neq},\cA\nabla_L(\widetilde{U}_{0}\cdot\nabla_L U^2_{\neq})\r_{H^N}\right|&\lesssim \norm{\nabla_L\cA  Q_{\neq}}_{L^2_tH^N}\norm{(\widetilde{Q},\widetilde{W})_{0}}_{L^\infty_t H^N}\norm{\cA Q_{\neq}}_{L^2_tH^N}\\
    &\lesssim (\nu^{-\frac23}\eps)\eps^2.
\end{align}
The remaining simple zero interactions \eqref{eq:neq_Qtransport_bd3} and \eqref{eq:neq_Qtransport_bd4} are treated as follows: Using \eqref{eq:simple0_conversion} and \eqref{eq:reconstruction_m}, we bound
\begin{align}
    \int_0^T \left|  \l \cA\nabla_LQ_{\neq},\cA\nabla_L(U^2_{\neq}\de_y \widetilde{U}^2_{0})\r_{H^N}\right|\lesssim \norm{\nabla_L\cA  Q_{\neq}}_{L^2_tH^N}\norm{\cA Q_{\neq}}_{L^2_tH^N}\norm{\widetilde{Q}_0}_{L^\infty_t H^N}\lesssim (\nu^{-\frac23}\eps)\eps^2,
\end{align}
and similarly, again invoking \eqref{eq:A-product-bd} when $\nabla_L$ falls on  $U^3_{\neq}$,
\begin{align} 
    \int_0^T \left|  \l \cA\nabla_LQ_{\neq},\cA\nabla_L(U^3_{\neq}\de_z \widetilde{U}^2_{0})\r_{H^N}\right|\lesssim \norm{\nabla_L\cA  Q_{\neq}}_{L^2_tH^N}\norm{\cA(Q,W)_{\neq}}_{L^2_tH^N}\norm{\widetilde{Q}_0}_{L^\infty_t H^N} \lesssim (\nu^{-\frac23}\eps)\eps^2.
\end{align}
Finally, the double zero interactions in \eqref{eq:neq_Qtransport_bd5} are treated as in \eqref{eq:doublezerosetim_simplezero}:
using \eqref{eq:reconstruction_m} and Proposition \ref{prop:doublezero_estimates} we have that for that for $j\in\{1,3\}$
\begin{align}
    \int_0^T \left|  \l \cA\nabla_LQ_{\neq},\cA\nabla_L(\overline{U}^{j}_{0}\de_{j} U^2_{\neq})\r_{H^N}\right|&\lesssim \norm{\nabla_L\cA Q_{\neq}}_{L^2_tH^N}\norm{ \overline{U}^{j}_0}_{L^\infty_t H^N}\norm{\cA Q_{\neq}}_{L^2_t H^N}\\
    &\quad +\norm{\nabla_L\cA Q_{\neq}}_{L^2_tH^N}\norm{\de_y \overline{U}^{j}_0}_{L^2_tH^N}\norm{\cA Q_{\neq}}_{L^\infty_t H^N}\\
    &\lesssim (\nu^{-\frac23}\eps + \nu^{-\frac32}\eps^2 +\alpha^{-\frac13}\nu^{-\frac83}\eps^3)\eps^2.
\end{align}
This concludes the proof.
\end{proof}

The last two terms in \eqref{eq:def_NE} can be estimated at once, as we show in the result below. This is the last step in our proof.
\begin{lemma}\label{lem:neq_Wtransport}
Under the assumptions of Theorem \ref{thm:bootstrap} we have for $r\in\{1,3\}$ that
\begin{equation}
\begin{aligned}
    \int_0^T\abs{\l \cA W_{\neq},\cA \abs{\nabla_L}\de_{4-r}\cT_{\neq}(U,U^r)\r_{H^N}} &= \int_0^T\abs{\l \cA|\nabla_L|W_{\neq},\cA \de_{4-r}(U\cdot\nabla_LU^{r})_{\neq}\r_{H^N}}\\
    &\lesssim (\nu^{-\frac56}\eps +\alpha^{-\frac13}\nu^{-\frac83}\eps^3)\eps^2.
\end{aligned}    
\end{equation}
\end{lemma}

\begin{proof}
We prove this lemma by establishing the following estimates:
\begin{align}
    \int_0^T\left|  \l \cA|\nabla_L|W_{\neq},\cA\de_{4-r}(U^{j}_{\neq}\de_{j} U^{r}_{\neq})\r_{H^N} \right|&\lesssim (\nu^{-\frac23}\eps)\eps^2,\qquad j\in\{1,3\},\label{eq:neq_Wtransport_bd1}\\ 
    \int_0^T\left|  \l \cA|\nabla_L|W_{\neq},\cA\de_{4-r}(U^2_{\neq}\de_y^L U^{r}_{\neq})\r_{H^N} \right|&\lesssim (\nu^{-\frac56}\eps)\eps^2,\label{eq:neq_Wtransport_bd2}\\
    \int_0^T\left|  \l \cA|\nabla_L|W_{\neq},\cA\de_{4-r}(\widetilde{U}_{0}\cdot\nabla_L U^{r}_{\neq})\r_{H^N} \right|&\lesssim (\nu^{-\frac23}\eps)\eps^2,\label{eq:neq_Wtransport_bd3}\\ 
    \int_0^T\left|  \l \cA|\nabla_L|W_{\neq},\cA\de_{4-r}(U^{j}_{\neq}\de_{j} \widetilde{U}^{r}_{0})\r_{H^N} \right|&\lesssim (\nu^{-\frac23}\eps)\eps^2,\qquad j\in\{2,3\},\label{eq:neq_Wtransport_bd4}\\
    \int_0^T\left|  \l \cA|\nabla_L|W_{\neq},\cA\de_{4-r}(\overline{U}^{j}_{0}\de_{j} U^{r}_{\neq})\r_{H^N} \right|&\lesssim (\nu^{-\frac23}\eps)\eps^2,\qquad j\in\{1,3\},\label{eq:neq_Wtransport_bd5}\\ 
    \int_0^T\left|  \l \cA|\nabla_L|W_{\neq},\cA\de_{4-r}(U^2_{\neq}\de_y \overline{U}^{r}_{0})\r_{H^N} \right|&\lesssim (\nu^{-\frac32}\eps^2+\alpha^{-\frac13}\nu^{-\frac83}\eps^3)\eps^2.\label{eq:neq_Wtransport_bd6}
\end{align}

To obtain \eqref{eq:neq_Wtransport_bd1}, we use \eqref{eq:A-Sobbd} and \eqref{eq:reconstruction_m} to conclude that for $j,r\in\{1,3\}$ there holds that
\begin{align}
    &\int_0^T \left| \l \cA|\nabla_L|W_{\neq},\cA\de_{4-r}(U^{j}_{\neq}\de_{j} U^{r}_{\neq})\r_{H^N}\right|\\
    &\qquad\leq \int_0^T \norm{\nabla_L\cA W_{\neq}}_{H^N}\e^{\delta \nu^{\frac13}t}\left(\norm{\de_{4-r}U^{j}_{\neq}}_{H^N}\norm{\de_{j} U^{r}_{\neq}}_{H^N}+\norm{U^{j}_{\neq}}_{H^N}\norm{\de_{4-r}\de_{j} U^{r}_{\neq}}_{H^N}\right)\\
    &\qquad\lesssim \norm{\nabla_L\cA  W_{\neq}}_{L^2_tH^N}\norm{\cA(Q,W)_{\neq}}_{L^\infty_t H^N}\norm{\cA(Q,W)_{\neq}}_{L^2_t H^N}\\
    &\qquad\lesssim (\nu^{-\frac23}\eps)\eps^2.
\end{align}
Analogously, employing \eqref{eq:extratimedecay} and \eqref{eq:absorption_nablaL} we bound the $U^2_{\neq}\de_y^LU^{r}_{\neq}$ interaction in \eqref{eq:neq_Wtransport_bd2} as 
\begin{align} 
    &\int_0^T \left| \l \cA|\nabla_L|W_{\neq},\cA\de_{4-r}(U^2_{\neq}\de_y^L U^{r}_{\neq})\r_{H^N}\right|\\
    &\qquad\lesssim \norm{\nabla_L\cA  W_{\neq}}_{L^2_tH^N}\norm{\sqrt{-\frac{\dot \cM}{\cM}}\cA Q_{\neq}}_{L^2_t H^N}\nu^{-\frac13}\norm{\cA(Q,W)_{\neq}}_{L^\infty_t H^N}\\
    &\qquad\lesssim (\nu^{-\frac56}\eps)\eps^2. 
\end{align}
Using also \eqref{eq:simple0_conversion} and Lemma \ref{lem:m-product}, we obtain \eqref{eq:neq_Wtransport_bd3} as
\begin{align}
    \int_0^T \left| \l \cA|\nabla_L|W_{\neq},\cA\de_{4-r}(\widetilde{U}_{0}\cdot \nabla_L U^{r}_{\neq})\r_{H^N}\right|&\lesssim \norm{\nabla_L\cA  W_{\neq}}_{L^2_tH^N}\norm{(\widetilde{Q},\widetilde{W})_{0}}_{L^\infty_t H^N}\norm{\cA(Q,W)_{\neq}}_{L^2_t H^N}\notag\\
    &\lesssim (\nu^{-\frac23}\eps)\eps^2.
\end{align}
For \eqref{eq:neq_Wtransport_bd4} we use \eqref{eq:reconstruction_m} and \eqref{eq:simple0_conversion} to deduce that for $j\in\{2,3\}$
\begin{align}
    \int_0^T \left| \l \cA|\nabla_L|W_{\neq},\cA\de_{4-r}(U^{j}_{\neq}\de_{j} \widetilde{U}^{r}_{0})\r_{H^N}\right|&\lesssim\norm{\nabla_L\cA  W_{\neq}}_{L^2_tH^N}\norm{\cA(Q,W)_{\neq}}_{L^2_t H^N}\norm{(\widetilde{Q},\widetilde{W})_{0}}_{L^\infty_t H^N}\notag\\
    &\lesssim(\nu^{-\frac23}\eps)\eps^2.
\end{align}
To handle the double zero mode interactions in \eqref{eq:neq_Wtransport_bd5} we use that $\de_x\overline{U}^{j}_0=\de_z\overline{U}^{j}_0=0$ for $j\in\{1,3\}$ to deduce using \eqref{eq:reconstruction_m} that
\begin{align}
    \int_0^T \left| \l \cA|\nabla_L|W_{\neq},\cA\de_{4-r}(\overline{U}^{j}_{0}\de_{j} U^{r}_{\neq})\r_{H^N}\right|&\lesssim \norm{\nabla_L\cA  W_{\neq}}_{L^2_tH^N}\norm{\overline{U}^{j}_0}_{L^\infty_t H^N}\norm{\cA(Q,W)_{\neq}}_{L^2_t H^N}\\
    &\lesssim(\nu^{-\frac23}\eps)\eps^2.
\end{align}
Similarly, we establish \eqref{eq:neq_Wtransport_bd6} by also using Proposition \ref{prop:doublezero_estimates}, which gives
\begin{align}
    \int_0^T \left| \l \cA|\nabla_L|W_{\neq},\cA\de_{4-r}(U^2_{\neq}\de_y \overline{U}^{r}_{0})\r_{H^N}\right|&\lesssim \norm{\nabla_L\cA  W_{\neq}}_{L^2_t H^N}\norm{\cA Q_{\neq}}_{L^\infty_t H^N}\norm{\de_y\overline{U}^{r}_0}_{L^2_t H^N}\notag\\
    &\lesssim(\nu^{-\frac32}\eps^2+\alpha^{-\frac13}\nu^{-\frac83}\eps^3)\eps^2.
\end{align}
The proof is over.
\end{proof}

\addtocontents{toc}{\protect\setcounter{tocdepth}{0}}
\section*{Acknowledgments}
The research of MCZ was partially supported by the Royal Society URF\textbackslash R1\textbackslash 191492 and EPSRC Horizon Europe Guarantee EP/X020886/1. ADZ and KW gratefully acknowledge support of the SNSF through grant PCEFP2\_203059.

\addtocontents{toc}{\protect\setcounter{tocdepth}{1}}

\bibliographystyle{amsplain}
\bibliography{bib}
\end{document}